\documentclass[11pt, a4paper,reqno]{amsart}

\usepackage{amsmath}
\usepackage{amsfonts}
\usepackage{esint}
\usepackage{a4wide}
\usepackage{amssymb}
\usepackage{amsthm}
\usepackage{palatino}

\usepackage{mathrsfs}

\usepackage[centertags]{amsmath}
\usepackage{eucal,dsfont,accents,bbm,amsfonts,amssymb,amsthm,amsopn}

\usepackage[usenames]{color}
\definecolor{citegreen}{rgb}{0,0.6,0}
\definecolor{refred}{rgb}{0.8,0,0}
\usepackage[colorlinks, citecolor=citegreen, linkcolor=refred]{hyperref}

\theoremstyle{plain}

\newtheorem{teo}{Theorem}[section]
\newtheorem{lemma}[teo]{Lemma}
\newtheorem{prop}[teo]{Proposition}
\newtheorem{cor}[teo]{Corollary}
\newtheorem{dfnz}{Definition}[section]

\newtheorem{ackn}{Acknowledgments\!}

\newtheorem{rem}[teo]{Remark}

\theoremstyle{definition}

\theoremstyle{remark}

\numberwithin{equation}{section}

\def\SS{{{\mathbb S}}}
\def\NN{{{\mathbb N}}}

\def\RR{{\mathbb R}}

\def\om{\omega}
\def\k{\kappa}

\def\a{\alpha}
\def\b{\beta}
\def\g{\gamma}
\def\z{\zeta}
\def\e{\eta}
\def\ep{\varepsilon}
\def\l{\lambda}

\def\Ric{{{Ric}}}

\title[Gradient Einstein solitons]{Gradient Einstein solitons}

\author[Giovanni Catino]{Giovanni Catino}
\address[Giovanni Catino]{Dipartimento di Matematica, Politecnico di Milano, Piazza Leonardo da Vinci 32, 20133 Milano, Italy}
\email[]{giovanni.catino@polimi.it}

\author[Lorenzo Mazzieri]{Lorenzo Mazzieri}
\address[Lorenzo Mazzieri]{Scuola Normale Superiore, P.za Cavalieri 7, 56126 Pisa, Italy}
\email{l.mazzieri@sns.it}

\begin{document}

\begin{abstract} In this paper we consider a perturbation of the Ricci solitons equation proposed by J. P. Bourguignon in~\cite{jpb1}. We show that these structures are more rigid then standard Ricci solitons. In particular, we prove that there is only one complete three--dimensional, positively curved, Riemannian manifold satisfying 
$$
Ric -\frac{1}{2} R \, g \, + \, \nabla^2 f \, = \,0\,,
$$
for some smooth function $f$. This solution is rotationally symmetric and asymptotically cylindrical and it represents the 
analogue of the Hamilton's cigar in dimension three. The key ingredient in the proof is the rectifiability of the potential function $f$. It turns out that this property holds also in the Lorentzian setting and for a more general class of structures which includes some gravitational theories.
\end{abstract}

\maketitle

%
%
%
%
%


\section{Introduction and statement of the results}

One of the most significant functional in Riemannian geometry is the Einstein--Hilbert action
$$
g \,\, \longmapsto \,\, \mathcal{E}(g) \,=\, \int_{M} R \, dV_{g} \,,
$$
where $M^n$, $n\geq 3$, is a $n$--dimensional compact differentiable manifold, $g$ is a Riemannian metric on $M^n$ and $R$ is its scalar curvature. It is well known that critical points of this functional on the space of metrics with fixed volume are Einstein metrics (see~\cite[Chapter~4]{besse}). In principle, it would be natural to use the associated (unnormalized) gradient flow 
\begin{equation}
\label{E-flow}
\partial_{t} g \,= \,-2\big(\Ric - \tfrac{1}{2}R\,g\big) \,
\end{equation}
to search for critical metrics. On the other hand, it turns out that such a flow is not parabolic. Hence, a general existence theory, even for short times, is not guaranteed by the present literature. This was one of the main reasons which led Hamilton to introduce the Ricci flow $\partial_t g=-2Ric$ in~\cite{hamilton1}. The Ricci flow has been studied intensively in recent years and plays a key role in Perelman's proof of the Poincar\'e conjecture (see~\cite{perel1},~\cite{perel3} and~\cite{perel2}). For an introduction to Ricci flow, we refer the reader to~\cite{chowluni}.

An important aspect in the treatment of the Ricci flow is the study of Ricci solitons, which generate self--similar solutions to the flow and often arise as singularity models. Gradient Ricci solitons are Riemannian manifolds satisfying 
$$
Ric + \nabla^2 f \, = \, \l \, g \, ,
$$
for some smooth function $f$ and some constant $\lambda \in \RR$. 
For a complete survey on this subject, which has been treated by many authors, we refer the interested reader to~\cite{cao3}~and~\cite{cao2}.

Motivated by the notion of Ricci solitons, it is natural to consider special solutions to the flow~\eqref{E-flow}, whose existence can be proved by ad hoc arguments. In particular, in this paper, we introduce the notion of {\em gradient Einstein solitons}. These are Riemannian manifolds 
satisfying 
$$
Ric -\frac{1}{2} R \, g \, + \, \nabla^2 f \, = \,  \l \, g \, ,
$$
for some smooth function $f$ and some constant $\lambda \in \RR$. As expected, Einstein solitons as well generate self--similar solutions to the Einstein flow~\eqref{E-flow}.

More in general, it is natural to consider on a Riemannian manifold $(M^n,g)$, $n\geq 3$, geometric flows of the type
\begin{equation}
\partial_{t} g \,=\, -2(\Ric - \rho R\,g)\,,
\end{equation}
for some $\rho \in \RR$, $\rho \neq 0$. In a forthcoming paper, we will develop the parabolic theory for these flows, which was first considered by Bourguignon in~\cite{jpb1}. We call these flows Ricci-Bourguignon flows. Here we just notice that we can prove short time existence for every $-\infty<\rho< 1/2(n-1)$. However, as far as the subject of our investigation are self-similar solutions, every value of $\rho$ is allowed. Associated to these flows, we have the following notion of gradient $\rho$--Einstein solitons.
\begin{dfnz}
\label{def-sol}
Let $(M^{n},g)$, $n\geq 3$, be a Riemannian manifold and let $\rho\in\RR$, $\rho\neq 0$. We say that $(M^{n},g)$ is a {\em gradient $\rho$--Einstein soliton} if there exists a smooth function $f:M^{n}\rightarrow \RR$, such that the metric $g$ satisfies the equation
\begin{equation}
\label{soliton}
\Ric + \nabla^{2}f = \rho R\,g + \lambda g\,,
\end{equation}
for some constant $\lambda\in\RR$. 
\end{dfnz}
We say that the soliton is {\em trivial} whenever $\nabla f$ is parallel. As usual, the $\rho$--Einstein soliton is {steady} for $\l=0$, {shrinking} for $\l>0$ and {expanding} for $\l<0$. The function $f$ is called a {\em $\rho$--Einstein potential} of the gradient $\rho$--Einstein soliton. 

Corresponding to special values of the parameter $\rho$, we refer to the $\rho$--Einstein solitons with different names, according to the Riemannian tensor which rules the flow. Hence, for $\rho=1/2$ we will have {\em gradient Einstein soliton}, for $\rho=1/n$ {\em gradient traceless Ricci soliton} and  for $\rho=1/2(n-1)$ {\em gradient Schouten soliton}. In the compact case, arguments based on the maximum principle yield the following triviality result (listed below as Corollary~\ref{cor-cpt}), for solitons corresponding to these special values of $\rho$. 
\begin{teo}
Every compact gradient Einstein, Schouten or traceless Ricci soliton is trivial.
\end{teo}
To deal with the noncompact case, it is useful to introduce the following notion of {\em rectifiability}. We say that a smooth function $f:M^n\rightarrow \RR$ is {\em rectifiable} in an open set $U \subset M^n$ if and only if $|\nabla f_{|U}|$ is constant along every regular connected component of the level sets of $f_{|U}$. In particular, it can be seen that $f_{|U}$ only depends on the signed distance $r$ to the regular connected component of some of its level sets. If $U= M^n$, we simply say that $f$ is rectifiable. Consequently, a gradient soliton is called rectifiable if and only if it admits a rectifiable potential function. 
The rectifiability turns out to be one of main property of the $\rho$--Einstein solitons, as we will show in the following Theorem.
\begin{teo}\label{teo-esol} Every gradient $\rho$--Einstein soliton is rectifiable.
\end{teo}
The reason for considering $n\geq 3$ in Definition~\ref{def-sol} and thus in Theorem~\ref{teo-esol} is that for $n=2$ equation~\eqref{soliton} reduces to the gradient Yamabe solitons equation (see~\cite{dasksesum}). The rectifiability of the potential function, in this case, follows easily from the structural equation and it has been used to describe the global structure of these solitons (see~\cite{caosunzhang} and~\cite{catmantmazz}).

It is worth noticing that Theorem~\ref{teo-esol} fails to be true in the case of gradient Ricci solitons. In fact, even though all of the easiest nontrivial examples -- such as the Gaussian soliton and the round cylinder in the shrinking case, or the Hamilton's cigar (also known in the physics literature as Witten's black hole) and the Bryant soliton in the steady case -- are rectifiable, it is easy to check, for instance, that the Riemannian product of rectifiable steady gradient Ricci solitons gives rise to a new steady soliton, which is generically not rectifiable.  

A well known claim of Perelman~\cite{perel1}, concerning gradient steady Ricci solitons, states that in dimension $n=3$ the Bryant soliton is the only complete noncompact gradient steady Ricci soliton with positive sectional curvature. Despite some recent important progresses, it remains a big challenge to prove this claim. Here, we notice that the rectifiabilty of the Ricci potential would imply the Perelman's claim. In this direction we have the following result.
\begin{teo}
\label{teo-erot1}
Up to homotheties, there exists only one three--dimensional gradient steady $\rho$--Einstein soliton with $\rho < 0$ or $\rho \geq 1/2$ and positive sectional curvature, namely the rotationally symmetric one constructed in Theorem~\ref{teo-warp-steady}.
\end{teo}
Theorem~\ref{teo-erot1} gives further evidences of the validity of Perelman's claim and could possibly be used to prove stability results for the Bryant soliton in the class of three-dimensional gradient steady Ricci solitons with positive sectional curvature. For $\rho = 1/2$, the only admissible three-dimensional gradient steady Einstein soliton with positive sectional curvatures turns out to be asymptotically cylindrical with linear volume growth. In other words, this soliton is the natural generalization of the two--dimensional Hamilton's cigar and we decided to call it {\em Einstein's cigar}. An immediate consequence of Theorem~\ref{teo-erot1} is the three--dimensional analogue of the Hamilton's uniqueness result for complete noncompact gradient steady Ricci solitons with positive curvature in dimension two (see~\cite{hamilton5}).  
\begin{cor} 
Up to homotheties, the only complete three--dimensional gradient steady Einstein soliton with positive sectional curvature is the Einstein's cigar.
\end{cor}
Among all the $\rho$--Einstein solitons, a class of particular interest is given by gradient Schouten solitons, namely Riemannian manifolds satisfying 
$$
Ric + \nabla^2 f \, = \, \frac{1}{2(n-1)} R \, g \, + \, \l \, g \, ,
$$
for some smooth function $f$ and some constant $\l\in\RR$.
Exploiting the rectifiability obtained in Theorem~\ref{teo-esol}, it is possible to achieve some classification results for this class of metrics. In the steady case, we can prove the following triviality result, which holds true in every dimension without any curvature assumption.
\begin{teo}
\label{teo-sssteady}
Every complete gradient steady Schouten soliton is trivial, hence Ricci flat.
\end{teo}
In particular, every complete three--dimensional gradient steady Schouten soliton is isometric to a quotient of $\RR^3$.
In analogy with Perelman's classification  of three-dimensional gradient shrinking Ricci solitons~\cite{perel1}, subsequently proved without any curvature assumption in~\cite{caochenzhu}, we have the following theorem. 
\begin{teo}
\label{teo-ssshrinking}
Every complete three--dimensional gradient shrinking Schouten soliton is isometric to a finite quotient of either $\SS^3$ or $\RR^3$ or $\RR \times \SS^2$.
\end{teo}

The plan of the paper proceeds as follows. In Section~\ref{sect-rect} we introduce the notion of {\em generalized Ricci potential}, which extends the concept of potential function for gradient $\rho$--Einstein soliton. In particular, these structures include some interesting examples of gravitational theories in Lorentzian setting (for instance, see~\cite{will}). We then prove local rectifiability for the subclass of {\em nondegenerate} generalized Ricci potential in the sense of Definition~\ref{def-nondeg} below. Finally, we describe the geometric properties of the regular connected components of their level sets. 

In Section~\ref{sect-esol}, after proving some triviality results for compact gradient $\rho$--Einstein solitons, we prove Theorem~\ref{teo-esol} (listed below as Theorem~\ref{teo-esolrect}) and we exploit the rectifiability to show the rotational symmetry of complete noncompact gradient $\rho$--Einstein solitons with positive sectional curvature under suitable assumptions (see Theorems~\ref{teo-rot} for the three--dimensional case and Theorem~\ref{teo-lcf} for the locally conformally flat case in every dimension). 

In Section~\ref{sect-warped}, motivated by the results obtained in Section~\ref{sect-esol}, we study complete simply connected gradient $\rho$--Einstein solitons, which are {\em warped products with canonical fibers}. 
In the complete noncompact case, we have that either the solution splits a line or it has positive sectional curvature everywhere. In the first case, we have that the soliton must be homothetic to either a round cylinder $\RR \times \SS^{n-1}$, or to the hyperbolic cylinder $\RR \times \mathbb{H}^{n-1}$ or to the flat $\RR^n$, as it is proven in Theorem~\ref{teo-warped-cyl}. In the case where the soliton has positive sectional curvature, we only focus on the steady case and we prove in Theorem~\ref{teo-warp-steady} some existence ($\rho < 1/2(n-1)$ or $\rho \geq 1/(n-1)$) and non existence results ($1/2(n-1) \leq \rho <1/(n-1)$). As a consequence of Theorem~\ref{teo-warp-steady} and the results in Section~\ref{sect-esol}, we obtain Theorem~\ref{teo-erot1} (listed below as Corollary~\ref{cor-rot}) and Corollary~\ref{cor-lcf}, which gives the classification of complete $n$--dimensional locally conformally flat gradient steady $\rho$--Einstein solitons with positive sectional curvature. In Proposition~\ref{prop-asymp}, we describe the asymptotic behavior of the solutions constructed in Theorem~\ref{teo-warp-steady}. In particular, it turns out that for $\rho=1/(n-1)$ the rotationally symmetric steady soliton is asymptotically cylindrical and provides the $n$--dimensional generalization of the Hamilton's cigar. We refer to these solutions as {\em cigar--type solitons}. 

In Section~\ref{Schouten solitons}, we focus on the case of Schouten solitons, which corresponds to $\rho = 1/2(n-1)$ and we prove Theorems~\ref{teo-sssteady} and \ref{teo-ssshrinking}, listed below as Theorem~\ref{teo-ssteady} and Theorem~\ref{teo-sshrinking}, respectively. Finally, in Section~\ref{sect-open}, we list some open questions and concluding remarks.

\

\

\begin{ackn} The authors are partially supported by the Italian project FIRB--IDEAS ÒAnalysis and BeyondÓ. They wish to thank Carlo Mantegazza for several interesting comments and discussions.
\end{ackn}

\

\section{Rectifiability and generalized Ricci potentials}
\label{sect-rect}

In this section we prove a local version of the rectifiability for a wide class of structures, which includes gradient $\rho$--Einstein solitons. To describe this class, we introduce the notion of generalized Ricci potential.

\begin{dfnz}
\label{def-genpot} Let $(M^{n},g)$, $n\geq 3$, be a Riemannian manifold and let $f:M^{n}\rightarrow \RR$ be a smooth function on it. We say that $f$ is a {\em generalized Ricci potential} for $(M^{n},g)$ around a regular connected component $\Sigma_{c}$ of the level set $\{f=c\}$ if there exist an open neighborhood $U$ of $\Sigma_{c}$ and smooth functions $\a,\b,\g,\z,\e : f(U)\rightarrow \RR$, such that the metric $g$ satisfies the equation
\begin{equation}\label{grp}
\Ric + \a \nabla^{2}f \,=\, \b df\otimes df + \g R\,g + \z g + \e P \,, \quad\quad\mbox{on}\,\,U\,,
\end{equation}
for some symmetric, parallel $(2,0)$--tensor $P$, such that $P(\nabla f,\cdot)=0$ and $P \circ \Ric = \Ric \circ P$. 
\end{dfnz}
We present now some examples of generalized Ricci potential. 
\begin{enumerate}
\item Gradient Ricci solitons, corresponding to $(\a,\b, \g, \zeta, \eta) = (1,0, 0, \l ,0)$, for some $\l\in \RR$.
\smallskip
\item Gradient $\rho$--Einstein solitons, corresponding to $(\a,\b, \g, \zeta, \eta) = (1,0,\rho,\l,0)$, for some $\l,\rho\in \RR$.
\smallskip
\item Quasi--Einstein metrics (see~\cite{mmancatmazrim} and~\cite{HePetWylie}), corresponding to $(\a,\b, \g, \zeta, \eta) = (1, \mu,0 , \l ,0)$, for some $\l,\mu\in \RR$.
\smallskip
\item Fischer--Marsden metrics (see~\cite{fishmars} and~\cite{koba2}), wherever the potential function is different form zero. These metrics satisfy the equation 
$$
f \, Ric - \nabla^2 f \, = \, \frac{f}{n-1} R \, g \, .
$$
Hence, where $f\neq 0$, they correspond to $(\a,\b, \g, \zeta, \eta) = (-1/f,0,1/(n-1),0,0)$.
\smallskip
\item Solutions to vacuum field equations in Lorentzian setting induced by actions of the following type
\begin{equation}
\label{action}
S(g,f) \,=\, \int_{M} \big( a(f) R + b(f) |\nabla f|^{2} \big) \,dV_{g} \,,
\end{equation}
where $a$ and $b$ are functions of the scalar field $f$. The associated Euler equations, if $a,b\neq 0$, are given by
\begin{eqnarray*}
Ric -\frac{1}{2}R\,g &=& \frac{a''-b}{a}\big(df \otimes df -\tfrac{1}{2}|\nabla f|^{2} g \big)- \frac{a''}{2a}|\nabla f|^{2} g + \frac{a'}{a}\big( \nabla^{2}f - \Box f\,g\big) \,,\\
\Box f &=& -\frac{1}{2} \frac{\big((a')^{2}-\tfrac{n-2}{n-1}ab \big)' }{ (a')^{2}-\tfrac{n-2}{n-1}ab }\,|\nabla f|^{2} \,,
\end{eqnarray*}
where $\Box f = g^{ij}\nabla_{i}\nabla_{j}f$. The simplified equation reads
$$
Ric -\frac{a'}{a}\nabla^{2}f \,\,=\,\, \frac{a''-b}{a}\,df \otimes df +\frac{a'b'-2a''b +(a')^{2}(b/a)}{2(n-2)b^{2}+2(n-1)a'b'-4(n-1)a''b}\,R g \,.
$$
Hence, where $a,b\neq 0$, these solutions correspond to a generalized Ricci potential with 
$$
(\a,\b, \g, \zeta, \eta) \,=\, \big( -\tfrac{a'}{a},\tfrac{a''-b}{a}, \tfrac{a'b'-2a''b +(a')^{2}(b/a)}{2(n-2)b^{2}+2(n-1)a'b'-4(n-1)a''b}, 0 ,0 \big)\,.
$$

\smallskip
\item[(5-bis)] Solutions to the vacuum field equations in Bergmann--Wagoner--Nordtvedt theory of gravitation (for an overview see~\cite{will})
\begin{equation*}
\label{action}
S(g,f) \,=\, \int_{M} \Big( f\, R - \frac{\om(f)}{f}|\nabla f|^{2} \Big) \,dV_{g} \,,
\end{equation*}
where $\om$ is a smooth function of the scalar field $f$. This is a particular case of Example (5) with $a(f)=f$ and $b(f)=-\om(f)/f$, we assume $a,b\neq 0$. The associated Euler equations are given by
\begin{eqnarray*}
\label{meq}
Ric-\frac{1}{2}R\,g &=& \frac{\om}{f^2}\big( df \otimes df - \tfrac{1}{2}|\nabla f|^2\,g \big) + \frac{1}{f} \big( \nabla^2 f - \Box f \,g \big) \,, \\ \label{seq}
\Box f &=& -\frac{\om'}{3+2\om} |\nabla f|^2 \,,
\end{eqnarray*}
where $\om'=d\om/d f$. A simple computation implies the following structure equation for the metric $g$
$$
Ric-\frac{1}{f}\nabla^{2}f \,\,=\,\, \frac{\om}{f^2} df\otimes df - \frac{\om' f}{(n-2)(3+2\om)\om -2(n-1) \om' f}R\,g \,. 
$$
These solutions correspond to a generalized Ricci potential with 
$$
(\a,\b, \g, \zeta, \eta) \,=\, \big( -\tfrac{1}{f},\tfrac{\om}{f^{2}}, - \tfrac{\om' f}{(n-2)(3+2\om)\om -2(n-1) \om' f}, 0 ,0 \big)\,.
$$
\end{enumerate}

For what follows, it is also convenient to introduce a notion of nondegeneracy for generalized Ricci potential. The motivation for the following definition comes from Theorem~\ref{teo-rectg} below.
\begin{dfnz}
\label{def-nondeg}
In the same setting as in Definition~\ref{def-genpot}, we say that a {\em generalized Ricci potential} $f$ is {\em nondegenerate} around $\Sigma_{c}$ if the following conditions are satisfied on $U$
\begin{eqnarray}
\label{nd1}&\a \neq 0\,,&\\
\label{nd2}&\a^{2}-\a'- \b \neq 0\,,&\\
\label{nd3}& \big(\tfrac{2\a\a'-\a''-\b'}{\a^{2}-\a'-\b}+\tfrac{2\b}{\a}\big)\big(\tfrac{1-2(n-1)\g}{2}\big) 
-\tfrac{(1-n\g)(\a'+\b)+\a^{2}\g}{\a} \neq 0 \,,&
\end{eqnarray}  
where we denoted by $(\cdot)'$ the derivative with respect to the $f$ variable.
\end{dfnz}
We notice that the smooth functions $\zeta$ and $\eta$ are not involved in the nondegeneracy conditions. We also observe that the metrics in the examples $(1), (3)$ and $(4)$, give rise to {\em degenerate} generalized Ricci potentials, whereas examples $(5)$ and $(5-bis)$ are generically {\em nondegenerate} generalized Ricci potentials. Moreover, it is not difficult to check that if we start with a generalized Ricci potential $f$ as in Definition~\ref{def-genpot} and we change the metric $g$ into $\tilde{g} = \phi^2 \, g$, for every positive smooth function $\phi: f(U) \rightarrow \RR$, we have that $f$ still remains a generalized Ricci potential for $\tilde{g}$ with suitably modified coefficients. On the other hand, we conjecture that the condition of being a degenerate Ricci potential is stable under this class of conformal changes (so far, we have evidences of this fact in the case of gradient Ricci solitons). We are now in the position to state the main theorem of this section.
\begin{teo}\label{teo-rectg} Let $(M^{n},g)$, $n\geq 3$, be a Riemannian manifold and let $f$ be a {\em nondegenerate generalized Ricci potential} for $(M^{n},g)$ around a regular connected component $\Sigma_{c}$ of the level set $\{f=c\}$. Then, there exists an open neighborhood $U$ of $\Sigma_{c}$ where $f$ is rectifiable.
\end{teo}
\begin{proof}
We start our analysis by proving a series of basic identities for $f$. The notations adopted are the same as in Definition~\ref{def-genpot}. 

\begin{lemma}\label{lemg} Let $(M^{n},g)$, $n\geq 3$, be a Riemannian manifold and let $f$ be a {\em nondegenerate generalized Ricci potential} for $(M^{n},g)$ around a regular connected component $\Sigma_{c}$ of the level set $\{f=c\}$. Then, there exist an open neighborhood $U$ of $\Sigma_{c}$ where the following identities hold
\begin{eqnarray}
\label{eq1}&\Delta f = \frac{\b}{\a}|\nabla f|^{2}+\frac{n\g-1}{\a}R+\frac{n\z}{\a}+\frac{\e}{\a} \, trP\,,&\\
\label{eq2}&\frac{1-2(n-1)\g}{2} \nabla R = \frac{\a^{2}-\a'-\b}{\a} \Ric(\nabla f,\,\cdot\,)+\frac{(n-1)(\a\g'-\a'\g-\g\b)+\a'+\b}{\a} R\nabla f + \sigma\nabla f \,,&\\\nonumber
&\nabla_{c}R_{ab}-\nabla_{b}R_{ac} = \g\big(g_{ab}\nabla_{c}R -g_{ac}\nabla_{b}R \big) + \frac{\b+\a'}{\a}\big(R_{ab}\nabla_{c}f - R_{ac}\nabla_{b}f\big) -\a \,R_{cabd}\nabla_{d}f&\\
\label{eq3}& \quad\,\,\,+\e'\big(P_{ab}\nabla_{c}f-P_{ac}\nabla_{b}f\big) +\xi\big(g_{ab}\nabla_{c}f-g_{ac}\nabla_{b}f\big)\,,&\\
\label{eq4}&\nabla_{b}R \,\nabla_{c}f = \nabla_{c} R \, \nabla_{b} f \,,&
\end{eqnarray}
where $trP=g^{ab}P_{ab}$ is the constant trace of the tensor $P$, $\sigma:f(U)\rightarrow \RR$ and $\xi:M^{d}\rightarrow \RR$ are smooth functions and we denoted by $(\cdot)'$ the derivative with respect to the $f$ variable.
\end{lemma}

\begin{proof}
The proof is divided in four steps, corresponding to each of the four identities.

{\bf Equation~\eqref{eq1}.} We simply contract equation~\eqref{grp}.

{\bf Equation~\eqref{eq2}.} Taking the divergence of the structural equation~\eqref{grp}, one has
\begin{eqnarray*}
0&=& \tfrac{1}{2}\nabla R + \a\nabla \Delta f + \a \,\Ric(\nabla f,\,\cdot\,)+\a'\nabla^{2}f(\nabla f,\,\cdot\,)\\
&&-\,\b\Delta f \nabla f-\b\nabla^{2}f(\nabla f, \,\cdot\,) -\b'|\nabla f|^{2}\nabla f\\
&&-\,\g\nabla R-\g'R\nabla f -\z'\nabla f\,,
\end{eqnarray*}
where we used Schur lemma $2\,\hbox{div }\Ric=dR$, the commutation formula for the covariant derivatives and the facts that $\nabla P=0$ and $P(\nabla f,\,\cdot\,)=0$. Using equation~\eqref{eq1} and observing that $trP$ is a constant function, we have
\begin{eqnarray*}
0 &=& \tfrac{2(n-1)\g-1}{2}\nabla R +(\b+\a')\nabla^{2}f(\nabla f,\,\cdot\,)+\a\,\Ric(\nabla f,\,\cdot\,)\\
&&-\,\tfrac{\b(\a'+\b)}{\a}|\nabla f|^{2}\nabla f+\tfrac{(n-1)\a\g'-n\a'\g+\a'-(n\g-1)\b}{\a}R\nabla f + \sigma_{2} \nabla f\,,
\end{eqnarray*}
where $\sigma_{2}:f(U)\rightarrow \RR$ is some function of $f$. Notice that we have used the nondegeneracy condition~\eqref{nd1} $\a\neq0$. Using again equation~\eqref{grp} to substitute the Hessian term $\nabla^{2}f$, we get
\begin{eqnarray*}
0 &=& \tfrac{2(n-1)\g-1}{2}\nabla R +\tfrac{\a^{2}-\a'-\b}{\a}\,\Ric(\nabla f,\,\cdot\,)\\
&&+\tfrac{(n-1)(\a\g'-\a'\g-\g\b)+\a'+\b}{\a} R\nabla f + \sigma_{3} \nabla f\,,
\end{eqnarray*}
where $\sigma_{3}:f(U)\rightarrow \RR$ is some function of $f$.

{\bf Equation~\eqref{eq3}.} Taking the covariant derivative of equation~\eqref{grp}, we get
\begin{eqnarray*}
\nabla_{c}R_{ab}-\nabla_{b}R_{ac} & = & -\a\big( \nabla_{c}\nabla_{b}\nabla_{a}f-\nabla_{b}\nabla_{c}\nabla_{a}f\big)\\
&&+\,\g\big(g_{ab}\nabla_{c}R-g_{ac}\nabla_{b}R\big)\\
&&+\,(\b+\a')\big(\nabla_{c}\nabla_{a}f\nabla_{b}f-\nabla_{b}\nabla_{a}f\nabla_{c}f\big)\\
&&+\,(\z'+\g'R)\big(g_{ab}\nabla_{c}f-g_{ac}\nabla_{b}f\big)\\
&&+\,\e'\big(P_{ab}\nabla_{c}f-P_{ac}\nabla_{b}f\big)\\
&=& -\a\,R_{cbad}\nabla_{d}f+\g\big(g_{ab}\nabla_{c}R-g_{ac}\nabla_{b}R\big)\\
&&+\,\tfrac{\b+\a'}{\a}\big(R_{ab}\nabla_{c}f-R_{ac}\nabla_{b}f\big)\\
&&+\,\e'\big(P_{ab}\nabla_{c}f-P_{ac}\nabla_{b}f\big) +\xi_{1}\big(g_{ab}\nabla_{c}f-g_{ac}\nabla_{b}f\big)\,,
\end{eqnarray*}
for some smooth function $\xi_{1}:M^{d}\rightarrow\RR$. Notice that in the last equality we used again the commutation formula and equation~\eqref{grp}.

{\bf Equation~\eqref{eq4}.} Taking the covariant derivative of equation~\eqref{eq2}, one obtains
\begin{eqnarray*}
\tfrac{1-2(n-1)\g}{2}\nabla_{b}\nabla_{c}R &=& \Big[\tfrac{2\a\a'-\a''-\b'}{\a}-\tfrac{(\a^{2}-\a'-\b)(\a'-\b)}{\a^{2}}\Big]\,R_{cd}\nabla_{b}f\nabla_{d}f\\
&&+\,\tfrac{\a^{2}-\a'-\b}{\a} \,\nabla_{b}R_{cd}\nabla_{d}f+\tfrac{\a'-(n-1)\a'\g+\b-(n-1)\b\g}{\a}\,\nabla_{b}R \nabla_{c} f \\
&&+\,\xi_{2}\big(\nabla_{b}R\nabla_{c}f+\nabla_{b}f\nabla_{c}R \big)+\xi_{3}\nabla_{b}f\nabla_{c}f\\
&&+\,\xi_{4}\nabla_{b}\nabla_{c}f+\xi_{5}R_{bc}+\xi_{6}R_{bd}R_{cd}+\xi_{7}P_{bd}R_{cd}\,,
\end{eqnarray*}
for some smooth functions $\xi_{2},\dots,\xi_{7}:M^{d}\rightarrow\RR$. Notice that the last six terms of the right hand side are symmetric. Thus, we get
\begin{eqnarray*}
0 &=& \tfrac{1-2(n-1)\g}{2}\big(\nabla_{b}\nabla_{c}R-\nabla_{c}\nabla_{b}R\big)\\
&=& \Big[\tfrac{2\a\a'-\a''-\b'}{\a}-\tfrac{(\a^{2}-\a'-\b)(\a'-\b)}{\a^{2}}\Big] \big(R_{cd}\nabla_{b}f-R_{bd}\nabla_{c}f \big)\nabla_{d}f\\
&&-\,\tfrac{\a^{2}-\a'-\b}{\a}\big(\nabla_{c}R_{bd} - \nabla_{b}R_{cd} \big) \nabla_{d}f \\
&&-\,\tfrac{\a'-(n-1)\a'\g+\b-(n-1)\b\g}{\a}\big(\nabla_{c}R \nabla_{b} f-\nabla_{b}R\nabla_{c}f \big)\,.
\end{eqnarray*}
Substituting equation~\eqref{eq3} in the second term of the right hand side, we obtain  
\begin{eqnarray*}
0 &=& \Big[\tfrac{2\a\a'-\a''-\b'}{\a}+\tfrac{2(\a^{2}-\a'-\b)\b}{\a^{2}}\Big] \big(R_{cd}\nabla_{b}f-R_{bd}\nabla_{c}f \big)\nabla_{d}f \\
&&-\,\tfrac{\a'-n\a'\g+\b-n\b\g +\a^{2}\g}{\a}\big(\nabla_{c}R \nabla_{b} f-\nabla_{b}R\nabla_{c}f \big)\,.
\end{eqnarray*}
Now, to conclude, it is sufficient to substitute the first term of the right hand side using equation~\eqref{eq2}. Notice that in doing that we make use of the nondegeneracy condition~\eqref{nd2} $\a^{2}-\a'-\b\neq 0$. 
\begin{eqnarray*}
0 &=& \Big[\big(\tfrac{2\a\a'-\a''-\b'}{\a^{2}-\a'-\b}+\tfrac{2\b}{\a}\big)\big(\tfrac{1-2(n-1)\g}{2}\big) 
-\tfrac{(1-n\g)(\a'+\b)+\a^{2}\g}{\a}\Big]\big(\nabla_{c}R \nabla_{b} f-\nabla_{b}R \nabla_{c}f \big)\,.
\end{eqnarray*}
Equation~\eqref{eq4} follows now from the third nondegeneracy condition~\eqref{nd3}
$$
\big(\tfrac{2\a\a'-\a''-\b'}{\a^{2}-\a'-\b}+\tfrac{2\b}{\a}\big)\big(\tfrac{1-2(n-1)\g}{2}\big) 
-\tfrac{(1-n\g)(\a'+\b)+\a^{2}\g}{\a} \neq 0\,.
$$
This concludes the proof of the lemma.
\end{proof}

Let now $U$ be the neighborhood of a regular connected component $\Sigma_{c}$ of the level set $\{f=c \}$, where equations~\eqref{eq1}--\eqref{eq4} are in force and, by continuity, let $\delta_{0}$ be a positive real number such that the regular connected components $\Sigma_{c+\delta}$ of the level sets $\{ f=c+\delta\}$ are subsets of $U$, for every $0\leq|\delta|<\delta_{0}$. We are going to prove that $|\nabla f|$ is constant along every $\Sigma_{c+\delta}$, $0\leq|\delta|<\delta_{0}$. First of all, we notice that $R$ is constant along every $\Sigma_{c+\delta}$. Indeed,  from equation~\eqref{eq4}, for all $V\in T\Sigma_{c+\delta}$, one has
\begin{equation}\label{rcost}
\langle \nabla R,V\rangle \,\,|\nabla f|^{2} \, = \, \langle \nabla R,\nabla f\rangle \, \langle\nabla f,V\rangle\,=\,0\,.
\end{equation}
Moreover, from the structural equation~\eqref{grp}
\begin{eqnarray}\label{gfcost}
\langle \nabla|\nabla f|^{2},V\rangle &=& 2 \,\nabla^{2} f (\nabla f, V)\\\nonumber
&=& (\tfrac{2\b}{\a}|\nabla f|^{2}+\tfrac{2\g}{\a} R +\tfrac{2\z}{\a})\langle \nabla f,V\rangle+\tfrac{2\e}{\a} T(\nabla f,V)-\tfrac{2}{\a}\Ric(\nabla f,V) \\\nonumber
&=& -\tfrac{1-2(n-1)\g}{\a^{2}-\a'-\b}\langle \nabla R,V \rangle \,\,=\,\,0\,,
\end{eqnarray}
where in the last equality we have used equation~\eqref{eq2} together with the fact that $P(\nabla f,\,\cdot\,)=~0$. We point out that we made use of the nondegeneracy conditions~\eqref{nd1}, \eqref{nd2}. This concludes the proof of Theorem~\ref{teo-rectg}.
\end{proof}

As a direct consequence of Theorem~\ref{teo-rectg}, we prove the following proposition, which describes some remarkable geometric properties of the regular level sets of a nondegenerate generalized Ricci potential.
\begin{prop}\label{prop-level} Let $(M^{n},g)$, $n\geq 3$, be a Riemannian manifold and let $f$ be a {\em nondegenerate generalized Ricci potential} for $(M^{n},g)$ around a regular connected component $\Sigma_{c}$ of the level set $\{f=c\}$. Then, the following facts hold:
\begin{itemize}
\item[(i)] The scalar curvature $R$ and $|\nabla f|$ are constants along $\Sigma_{c}$.
\item[(ii)] The mean curvature $H$ of $\Sigma_{c}$ is constant.
\item[(iii)] The scalar curvature $R^{\Sigma}$ of $(\Sigma_{c},g)$, with the induced metric by $g$ on $\Sigma_{c}$, is constant.
\end{itemize}
In particular, in a neighborhood of $\Sigma_{c}$, the generalized Ricci potential $f$ and all the geometric quantities $R,|\nabla f|, H$ and $R^{\Sigma}$ only depend on the signed distance $r$ to $\Sigma_{c}$.
\end{prop}
\begin{proof}
As we have already observed in equations~\eqref{rcost} and \eqref{gfcost}, property (i) follows immediately. 
From this we deduce that, in a neighborhood $U$ of $\Sigma_{c}$ where equations~\eqref{eq1}--\eqref{eq4} are in force, the generalized Ricci potential $f$ only depends on the signed distance $r$ to the hypersurface $\Sigma_{c}$. In fact, since $\nabla r$ coincides with the unit normal vector $\nu=\nabla f/|\nabla f|$, one has $df=|\nabla f|dr=f' dr$, where $f'=df/dr$.  Moreover, if $\theta=(\theta^{1}\,\ldots,\theta^{n-1})$ are coordinates {\em adapted} to the hypersurface $\Sigma_{c}$, we get
$$
\nabla^{2}f \,=\, \nabla df \,=\, f'' dr\otimes dr + f' \nabla^2 
r=\, f'' dr\otimes dr + \frac{f'}{2} \, \partial_r g_{ij} \,d\theta^i\otimes d\theta^j\,,
$$
as
$$
\Gamma_{rr}^r=\Gamma_{rr}^k=\Gamma_{ir}^r=0\,,\qquad
\Gamma_{ij}^r=- \frac{1}{2} \,\partial_r g_{ij}\,,\qquad
\Gamma_{ir}^k= \frac{1}{2} \, g^{ks}\partial_r g_{is}\,,
$$
for $i,j=1,\dots,n-1$.

To prove (ii), we recall that the second fundamental form $h$ 
verifies
$$
h_{ij} \,=\, \frac{(\nabla^{2} f)_{ij}}{|\nabla f|} \,=\,- \frac{R_{ij} -
  (\g R+\z) g_{ij}-\e P_{ij}}{\a|\nabla f|}\,,
$$
for $i,j=1,\dots,n-1$.
Thus, the mean curvature $H$ of $\Sigma_{c}$
satisfies 
\begin{eqnarray*}
H = g^{ij}\,h_{ij} &=& -\frac{R -R_{rr}-(n-1)(\g R+\z)-\e \tau + \e P_{rr}}{\a|\nabla f|} \\
&=& -\frac{(1-(n-1)\g)R -R_{rr}-(n-1)\z-\e tr(P)}{\a|\nabla f|}\,.
\end{eqnarray*}
Now, combining equation~\eqref{eq2} with property (i), it is easy to deduce that $R_{rr}$ is constant along $\Sigma_{c}$ and property (ii) follows.

In order to prove (iii), we consider the contracted 
{\em Riccati equation} (see~\cite[Chapter~1]{chowluni})
$$
|h|^{2} = - H' - R_{rr} \,
$$
and we deduce at once that the norm of the second fundamental form
$|h|^{2}$ is also constant on $\Sigma_{c}$. 
Now, from the {\em Gauss
equation} (see again~\cite[Chapter~1]{chowluni})
$$
R^{\Sigma} = R - 2R_{rr} -|h|^{2} +H^{2} \,,
$$
we conclude that the scalar curvature $R^{\Sigma}$ of the metric
induced by $g$ on $\Sigma_{c}$ is constant and property (iii) follows. It is now immediate to observe that in $U$ all the quantities $R,|\nabla f|,H$ and $R^{\Sigma}$ only depend on $r$.
\end{proof}
\begin{rem}
\label{GR}
We observe that all the computations in this section still remain true in Lorentzian or even semi-Riemannian setting. In particular, whenever the nondegeneracy conditions are satisfied by the coefficients involved in examples $(5)$ and $(5-bis)$ above, we have that the corresponding space--time solutions to the relativistic Einstein field equations are necessarily foliated by hypersurfaces with constant mean curvature and constant induced scalar curvature, about a regular connected component of a level set of the potential $f$ .
\end{rem}

\

\section{Gradient $\rho$--Einstein solitons}
\label{sect-esol}

We pass now to the analysis of gradient $\rho$--Einstein solitons. We recall that a gradient $\rho$--Einstein soliton is a Riemannian manifold $(M^{n},g)$, $n\geq 3$, 
endowed with a smooth function $f:M^{n}\rightarrow \RR$, such that the metric $g$ satisfies the equation
\begin{equation}\label{esol}
\Ric + \nabla^{2}f = \rho R\,g + \lambda g\,,
\end{equation}
for some constants $\rho,\lambda\in\RR$, $\rho\neq 0$. To see that gradient $\rho$--Einstein solitons generate self--similar solutions to the $\rho$--Einstein flow $\partial_{t} g = -2(Ric-\rho R\,g)$ it is sufficient to set $g(t) = (1-2\lambda t) \,\varphi_{t}^{*} (g)$, if $1-2\lambda t>0$, where $\varphi_{t}$ is the $1$-parameter family of diffeomorphisms generated by $Y(t) = \nabla f / (1-2\lambda t)$ with $\varphi_{0} = id_{M^{n}}$.

We start by focusing our attention on compact gradient $\rho$--Einstein solitons and we prove the following triviality result.
\begin{teo}\label{teo-cptesol} Let $(M^{n},g)$, $n\geq 3$, be compact gradient $\rho$--Einstein soliton. Then, the following cases occur.
\begin{itemize}
\item[(i)] If $\rho\leq 1/2(n-1)$, then either $\lambda>0$ and $R>0$ or the soliton is trivial.
\item[(i-bis)] If $\rho=1/2(n-1)$, then the soliton is trivial.
\item[(ii)] If $1/2(n-1)<\rho< 1/n$, then either $\lambda<0$ and $R<0$ or the soliton is trivial.
\item[(iii)] If $1/n \leq\rho$, then the soliton is trivial.
\end{itemize}
\end{teo}
\begin{proof} It follows from the general computation in Lemma~\ref{lemg}, that if equation~\eqref{esol} is in force, then we have
\begin{eqnarray}
&\Delta f = (n\rho-1)R+n\lambda\,,&\label{equ1}\\
&\big(1-2(n-1)\rho\big)\nabla R = 2 \Ric (\nabla f, \,\cdot\,)\,.&\label{equ2}
\end{eqnarray}
Since $M^{n}$ is compact, integrating equation~\eqref{equ1}, we obtain the identity
\begin{equation}\label{equi}
\lambda \,=\, \tfrac{1-n\rho}{n} \fint_{M} R\,dV_{g}\,,
\end{equation}
where $\fint_{M}R\,dV_{g}=\mathrm{Vol}_{g}(M)^{-1}\int_{M}R\,dV_{g}$.
Taking the divergence of equation~\eqref{equ2}, we obtain
\begin{equation}\label{equ3}
\big(1-2(n-1)\rho\big)\Delta R = \langle \nabla R,\nabla f\rangle + 2( \rho R^{2}-|\Ric|^{2}+\lambda R )\,. 
\end{equation}

{\bf Case (i): $\rho \leq 1/2(n-1)$.} Let $q$ be a global minimum point of the scalar curvature $R$. Then, from equation~\eqref{equ3}, one has
\begin{eqnarray*}
0 &\leq& \big(1-2(n-1)\rho\big)\Delta R _{|q} \,=\, 2( \rho R^{2}-|\Ric|^{2}+\lambda R )_{|q}\\
&\leq& 2 R_{|q}\big(\lambda -\tfrac{1-n\rho}{n} R\big)_{|q} \,,
\end{eqnarray*}
where in the last inequality we have used $|\Ric|^{2}\geq (1/n) R^{2}$. Since $R(p)\geq R(q)$ for all $p\in M^{n}$, then from~\eqref{equi} we deduce that
$$
\lambda \geq \tfrac{1-n\rho}{n} R_{|q}\,,
$$
with equality if and only if $R\equiv R_{|q}$. In this latter case equation~\eqref{equ1} implies that $\Delta f=0$, thus $f$ is constant and the soliton is trivial. On the other hand, the strict inequality implies $R_{|q}\geq 0$ which forces $\lambda>0$ and $R>0$.

{\bf Case (i-bis): $\rho = 1/2(n-1)$.} Assume that the soliton is not trivial. Then, by case (i), we can assume $R>0$. First of all, we notice that equation~\eqref{equ2} implies that $\nabla f/|\nabla f|$ is an eigenvector of the Ricci tensor with zero eigenvalue, i.e. $\Ric(\nabla f/|\nabla f|, \,\cdot\,)=0$ on $M^{n}$. In particular the following inequality holds $|\Ric|^{2}\geq R^{2}/(n-1)$. On the other hand, from equation~\eqref{equ3}, one has
$$
\langle \nabla R,\nabla f\rangle = 2|\Ric|^{2}-\tfrac{1}{n-1}R^{2}-2\lambda R \geq \tfrac{1}{n-1}R \big( R-2(n-1)\lambda \big)\,.
$$
Let $q$ be a global maximum point of the scalar curvature $R$. Then, since $R(q)>0$, we obtain 
$$
R_{|q} \leq 2(n-1) \lambda\,.
$$
Since $R(p)\leq R(q)$ for all $p\in M^{n}$, then from~\eqref{equi} we deduce that $\lambda \leq \tfrac{n-2}{2n(n-1)} R_{|q}\,,$
which contradicts the the positivity of the scalar curvature.

{\bf Case (ii): $1/2(n-1)<\rho < 1/n$.} Let $q$ be a global maximum point of the scalar curvature $R$. Then, from equation~\eqref{equ3}, one has
\begin{eqnarray*}
0 &\leq& \big(1-2(n-1)\rho\big)\Delta R _{|q} \,=\, 2( \rho R^{2}-|\Ric|^{2}+\lambda R )_{|q}\\
&\leq& 2 R_{|q}\big(\lambda -\tfrac{1-n\rho}{n} R\big)_{|q} \,,
\end{eqnarray*}
where in the last inequality we have used $|\Ric|^{2}\geq (1/n) R^{2}$. Since $R(p)\leq R(q)$ for all $p\in M^{n}$, then from~\eqref{equi} we deduce that
$$
\lambda \leq \tfrac{1-n\rho}{n} R_{|q}\,,
$$
with equality if and only if $R\equiv R_{|q}$. In this latter case equation~\eqref{equ1} implies that $\Delta f=0$, thus $f$ is constant and the soliton is trivial. On the other hand, the strict inequality implies $R_{|q}\leq 0$ which forces $\lambda<0$ and $R<0$.

{\bf Case (iii): $1/n \leq \rho $.} First of all, we notice that if $\rho=1/n$, then from equation~\eqref{equ1}, one has $\Delta f=n\lambda$ on $M^{n}$. This forces $\lambda=0$ and $f$ to be constant. On the other hand, if $1/n < \rho$, we integrate equation~\eqref{equ3} obtaining
\begin{eqnarray*}
0 &=& \int_{M}\langle \nabla R,\nabla f\rangle \,dV_{g} + 2\int_{M}( \rho R^{2}-|\Ric|^{2}+\lambda R )\,dV_{g} \\
&=& -\int_{M}R\Delta f \,dV_{g} + 2\int_{M}( \rho R^{2}-|\Ric|^{2}+\lambda R )\,dV_{g} \\
&=& \int_{M} \big[ \big(1-(n-2)\rho\big)R^{2}-2|\Ric|^{2}-(n-2)\lambda R\big] \,dV_{g} \\
&\leq& -\tfrac{(n-2)(n\rho-1)}{n} \int_{M} R^{2}\,dV_{g}- (n-2)\lambda \int_{M}R \,dV_{g}\,,
\end{eqnarray*} 
where we have used the inequality $|\Ric|^{2}\geq (1/n) R^{2}$. Substituting identity~\eqref{equi}, we get
\begin{eqnarray*}
0 \leq \Big(\fint_{M}R \,dV_{g}\Big)^{2}  - \fint_{M} R^{2}\,dV_{g}\leq 0\,,
\end{eqnarray*} 
by the Cauchy--Schwartz inequality. Hence, $R$ must be constant and the soliton must be trivial.
\end{proof}
We observe that the same statement as in case (i) was already known for compact gradient Ricci solitons (formally corresponding to $\rho=0$, see~\cite{hamilton9} and~\cite{ivey1}). An immediate consequence of Theorem~\ref{teo-cptesol} is the following corollary, concerning the most significant classes of $\rho$--Einstein solitons.
\begin{cor}
\label{cor-cpt}
Every compact gradient Einstein, Schouten or traceless Ricci soliton is trivial.
\end{cor}
To conclude the analysis in the compact case, we notice that compact gradient Schouten solitons appeared in~\cite[equation 2.12]{leltop} as a first characterization of the equality case in an optimal $L^{2}$--curvature estimate on manifolds with nonnegative Ricci curvature.

We turn now our attention to the case of general (possibly noncompact) gradient $\rho$--Einstein solitons. As an immediate consequence of Theorem~\ref{teo-rectg}, we have the following:

\begin{teo}\label{teo-esolrect} 
Every gradient $\rho$--Einstein soliton is rectifiable.
\end{teo}
\begin{proof} It is sufficient to check that the nondegeneracy conditions~\eqref{nd1}--\eqref{nd3} with $\alpha=1$, $\beta=0$ and $\gamma\equiv\rho\neq 0$ are satisfied everywhere. Hence, we can apply Theorem~\ref{teo-rectg}.
\end{proof}

In the previous section we have seen that if a Riemannian manifold $(M^{n},g)$ admits a nondegenerate generalized Ricci potential $f$, then, around every regular regular connected component of a level sets of $f$, the manifold is foliated by constant mean curvature hypersurfaces. Obviously, the same is true for gradient $\rho$--Einstein solitons. Moreover, in dimension $n=3$ this fact has immediate stronger consequences, which we summarize in the following theorem.

%
%
%
%

\begin{teo}\label{teo-rot} Let $(M^{3},g)$ be a three--dimensional gradient $\rho$--Einstein soliton with $\rho<0$ and $\lambda\leq 0$ or $\rho\geq 1/2$ and $\lambda\geq 0$. If $(M^{3},g)$ has positive sectional curvature, then it is rotationally symmetric. 
\end{teo} 
\begin{proof} We give the proof only in the case $\rho\geq 1/2$ and $\lambda\geq 0$. The proof of the other part of the statement follows with minor changes and it is left to the interested reader. 

First of all we notice that $g$ has positive sectional curvature if and only if the Einstein tensor $\Ric-(1/2)R\,g$ is negative definite. Hence, from the soliton equation, it follows that $f$ is a strictly convex function. In particular $M^{3}$ is diffeomorphic to $\RR^3$ and $f$ has at most one critical point. We claim that $f$ has exactly one critical point. In fact, by the strict convexity of $f$, we have that all of its level sets are compact. Now, if $f$ has no critical points, then the manifold would have two ends (see~\cite[Remark 2.7]{bishoponeill}). Since $\Ric\geq 0$, it would follow from Cheeger--Gromoll Theorem that the manifold splits a line, but this contradicts the strict positivity of the sectional curvature. Hence, the claim is proved. Let $O\in M^{3}$ be the unique critical point of $f$ and let $\Sigma \subset M^{3}\setminus\{O\}$ be a level set of $f$. Then $\Sigma$ is compact, regular, orientable and its second fundamental form is given by
$$
h_{ij} \,=\, -\frac{R_{ij} -
  (\rho R+\lambda) g_{ij}}{|\nabla f|}\,,
$$
for $i,j=1,2$. Since $\rho\geq 1/2$, $\lambda\geq 0$ and $g$ has positive sectional curvature, then $h_{ij}$ is positive definite. In particular $|h|^{2} < H^{2}$ and from Gauss equation we have
$$
R^{\Sigma} = R - 2R_{rr} -|h|^{2} +H^{2} >0 \,.
$$
Using Theorem~\ref{teo-esolrect} and Proposition~\ref{prop-level} we have that $(\Sigma,g^{\Sigma})$ has constant positive curvature. This implies that, up to a constant factor, $(\Sigma, g^{\Sigma})$ is isometric to $(\SS^{2}, g^{\SS^{2}})$ and on $M^{3}\setminus\{O\}$ the metric $g$ takes the form
$$
g \,=\, dr\otimes dr + \om(r)^{2}g_{\SS^{2}}\,,
$$
where $r(\cdot)=dist(O,\cdot)$ and $\om:\RR^{+}\rightarrow \RR^{+}$ is a positive smooth function.
\end{proof}

A possible extension of Theorems~\ref{teo-rot} above to higher dimensions may be obtained in the spirit of~\cite{caochen},~\cite{mancat1} and~\cite{mmancatmazrim}, under the additional hypothesy that the manifold is locally conformally flat. We notice that in this approach, the rectifiability is most of the time deduced as a consequence of the locally conformally flatness coupled with the soliton structure. In our case it would still be possible to proceed this way, however we will take advantage of the rectifiability provided by Theorem~\ref{teo-esolrect} to get a shortcut in the proof.

\begin{teo}\label{teo-lcf} Let $(M^{n},g)$ be a complete $n$--dimensional, $n\geq 4$, locally conformally flat gradient $\rho$--Einstein soliton with $\rho<0$ and $\lambda\leq 0$ or $\rho\geq 1/2$ and $\lambda\geq 0$. If $(M^{n},g)$ has positive sectional curvature, then it is rotationally symmetric.
\end{teo}
\begin{proof} 
We give the proof only in the case $\rho\geq 1/2$ and $\lambda\geq 0$. The proof of the other part of the statement follows with minor changes and it is left to the interested reader. 

First of all we notice that since $g$ is locally conformally flat and it has positive sectional curvature then the tensor $\Ric-(1/2)R\,g$ is negative definite. In fact, from the decomposition formula for the curvature tensor, it follows that 
$$
\lambda_{i}+\lambda_{j} \,> \,\frac{1}{n-1}\,R \,, 
$$
for every $i=1,\ldots,n$, where $\lambda_{i}$ are the eigenvalues of the Ricci tensor. Hence, from the soliton equation, it follows that $f$ is a strictly convex function. In particular $M^{n}$ is diffeomorphic to $\RR^n$ and $f$ has at most one critical point. We claim that $f$ has exactly one critical point. In fact, by the strict convexity of $f$, we have that all of its level sets are compact. Now, if $f$ has no critical points, then the manifold would have two ends (see~\cite[Remark 2.7]{bishoponeill}). Since $\Ric\geq 0$, it would follow from Cheeger--Gromoll Theorem that the manifold splits a line, but this contradicts the strict positivity of the sectional curvature. Hence, the claim is proved. Let $O\in M^{n}$ be the unique critical point of $f$ and let $\Sigma \subset M^{n}\setminus\{O\}$ be a level set of $f$.  To be definite, we choose the sign of $r$ in such a way that $f' = |\nabla f|$.  By Theorem~\ref{teo-esolrect} and Proposition~\ref{prop-level}, we also have that $f, |\nabla f|, R$ and $R^{\Sigma}$, which is the scalar curvature induced on the level sets of $f$, only depend on $r$. The proof follows the one in~\cite[Theorem 1.1]{mmancatmazrim}. With the same convention as in Proposition~\ref{prop-level}, the second fundamental form and the mean curvature of $\Sigma$ are given by
$$
h_{ij} \,=\, -\frac{R_{ij} - (\rho R+\lambda) g_{ij}}{|\nabla f|} \quad\quad \hbox{and} \quad\quad H \, = \, - \, \frac{(1-(n-1)\rho) R - R_{rr} - (n-1) \l}{|\nabla f|} \,,
$$
for $i,j=1,\dots, n-1$. We are going to prove that $(\Sigma, g^{\Sigma})$ is totally umbilic, namely 
$$
h_{ij}~-~ (H/(n-1) )~g_{ij} \,=\, 0 \,.
$$
In the spirit of~\cite[Theorem 1.1]{mmancatmazrim}, we introduce the Cotton tensor  
$$
C_{abc} = \nabla_{c} R_{ab} - \nabla_{b} R_{ac} -
\frac{1}{2(n-1)} \big( \nabla_{c} R \, g_{ab} - \nabla_{b} R \,
g_{ac} \big)\, ,
$$
for $a,b,c= 1, \ldots ,n$. Now, if we assume that the manifold is locally conformally flat, then the Cotton tensor is identically zero, since 
$$
0 \, = \, \nabla_d W_{abcd} \, = \, - \frac{n-3}{n-2} \, C_{abc} \, ,
$$
where $W$ is the Weyl tensor of $g$. Using Lemma~\ref{lemg} and the formul\ae \, for the second fundamental form and the mean curvature of $\Sigma$, and taking advantage of the rectifiability, it is straightforward to compute
$$
0 \, = \, C_{ijr} \, = \, \frac{|\nabla f|^2}{(n-2)} \, \Big( \, h_{ij} - \frac{H}{n-1} g_{ij}  \, \Big) \, ,
$$
for $i,j = 1, \dots, n-1$. Hence, the umbilicity is proven. From the Gauss equation (see also~\cite[Lemma 3.2]{caochen} for a
similar argument), one can see that the sectional curvatures of
$(\Sigma,g^{\Sigma})$ are
given by
\begin{eqnarray*}
R_{ijij}^{\Sigma} &=& R_{ijij} + h_{ii}h_{jj} - h_{ij}^{2}\\
&=& \, \tfrac{1}{n-2}\big(R_{ii}+R_{jj}\big) -
\tfrac{1}{(n-1)(n-2)}R + \tfrac{1}{(n-1)^{2}}H^{2}\\
&=& \, \tfrac{2}{n-2}R_{ii} - \tfrac{1}{(n-1)(n-2)}R +
\tfrac{1}{(n-1)^{2}}H^{2}\\
&=& - \tfrac{2}{(n-1)(n-2)}H|\nabla f| + \tfrac{2}{n-2}\lambda -
\tfrac{1-2(n-1) \rho}{(n-1)(n-2)}R + \tfrac{1}{(n-1)^{2}}H^{2}\,,
\end{eqnarray*}
for $i,j=1,\dots, n-1$, where in the second equality we used the decomposition formula for the Riemann tensor, the locally conformally flatness of $g$ and the umbilicity. Since, by Proposition~\ref{prop-level}, all the terms on the right hand side are constant on $\Sigma$, we obtain that the sectional curvatures of $(\Sigma,g^{\Sigma})$ are constant. The positivity follows from the Gauss equation, the umbilicity and the fact that $(M^n, g)$ has positive sectional curvature. It follows that, up to a constant factor only depending on $r$, $(\Sigma,g^{\Sigma})$ is isometric to $(\SS^{n-1}, g^{\SS^{n-1}})$. Hence, on $M^{n}\setminus \{O\}$ the metric $g$ takes the form
$$
g \,=\, dr\otimes dr + \om(r)^{2}g_{\SS^{n-1}}\,,
$$
where $r(\cdot)=dist(O,\cdot)$ and $\om:\RR^{+}\rightarrow \RR^{+}$ is a positive smooth function. In particular, this shows that $g$ is rotationally symmetric.
\end{proof}

\

\section{Warped product gradient $\rho$--Einstein solitons with canonical fibers}
\label{sect-warped}

In this section, motivated by Theorems~\ref{teo-rot}, \ref{teo-lcf}, we study complete simply connected gradient $\rho$--Einstein solitons $(M^{n},g)$, $n\geq 3$, which are {\em warped product with canonical fibers}. More precisely, we assume that $g$ is of the form
\begin{equation}\label{eqwarp}
g = dr \otimes dr + \om(r)^{2} g_{can} \quad\quad\hbox{in $M^{n}\setminus \Lambda$}\,,
\end{equation}
where $g_{can}$ is a constant curvature metric on a $(n-1)$--dimensional manifold, $r\in(r_{*},r^{*})$, $-\infty\leq r_{*}<r^{*}\leq +\infty$, the warping factor $\om:(r_{*},r^{*})\rightarrow \RR^{+}$ is a positive smooth function and $\Lambda\subset M^{n}$ consists of at most two points, depending on the behavior of $\om$ as $r\rightarrow r_{*}$ and $r\rightarrow r^{*}$.

The main focus of this section will be the analysis of gradient steady $\rho$--Einstein solitons with positive sectional curvature which are warped product with canonical fibers.

%

\begin{rem}
\label{rem-ricci}
It is worth pointing out that all the analysis of this section is consistent with the limit case of gradient Ricci solitons $(\rho = 0)$.  
\end{rem}
For notational convenience we set $m=n-1$. We agree that $\Ric_{can} = (m-1) \k \,g_{can}$, $\k\in\{-1,0,1\}$. In $M^{n}\setminus\Lambda$, the Ricci curvature and the scalar curvature of $g$ have the form
\begin{eqnarray*}
\Ric &=& -m \frac{\om''}{\om} dr\otimes dr + \big((m-1)(\k-(\om')^{2}) - \om\,\om''\big)\,g_{can}\,,\\
R &=& -2m\frac{\om''}{\om} + m(m-1)\,\frac{\k-(\om')^{2}}{\om^{2}}\,.
\end{eqnarray*}
Moreover, the Hessian of $f$ reads 
$$
\nabla^{2}f = f'' dr\otimes dr + \om \om' f' g_{can}\,.
$$
Hence, the soliton equation~\eqref{soliton} reduces to
\begin{equation*}
\begin{cases} f''\om^{2}-(m-2m\rho)\om\om''+m(m-1)\rho (\om')^{2}-\l\om^{2}-m(m-1)\rho\k = 0 \\
f'\om\om'  -(1-2m\rho)\om\om''-(m-1)(1-m\rho)(\om')^{2}-\l\om^{2}+(m-1)(1-m\rho)\k =0 \,.
\end{cases} 
\end{equation*}
Introducing the variables
$$
x = \om' \quad\quad \hbox{and} \quad\quad y = -\om f' 
$$
and the independent variable $t$, which satisfies $dt=(1/\om)dr$, one obtains the first--order system
\begin{equation}\label{system}
\begin{cases} 
(1-2m\rho)\,\dot{x} = (m-1)(1-m\rho)(\k-x^{2}) -xy -\l \om^{2} \\
(1-2m\rho) \,\dot{y} = -m(m-1)(1-(m+1)\rho)(\kappa-x^{2})+(1+m-4m\rho)xy+(m-1)\l\om^{2}\hspace{-0.45cm}\\
\hspace{1.69cm}\dot{\om} = x \om \, ,
\end{cases} 
\end{equation}
for every $t\in(t_{*},t^{*})$ where $t_{*}=\lim_{r\rightarrow r_{*}}t(r)$ and $t^{*}=\lim_{r\rightarrow r^{*}}t(r)$. In the system above $(\,\dot{ }\,)$ denotes the derivative with respect to the $t$ variable, and with a small abuse of notation we consider $\om$ as a function of $t$. It is immediate to see that the equilibrium points of this system in the $xy\om$--space are $P=(1,0,0)$ and $Q=(-1,0,0)$. We start with some general consideration about the interval of definition of the variables $r$ and $t$.

%

We observe that, if $M^{n}$ is compact, we have that $-\infty<r_{*}<r^{*}<+\infty$ and
$$
\lim_{r\rightarrow r_{*}} \om(r) \,=\, 0 \quad\quad\hbox{and} \quad\quad \lim_{r\rightarrow r_{*}} \om'(r) \,=\, 1\,,
$$
$$
\,\,\,\,\lim_{r\rightarrow r^{*}} \om(r) \,=\, 0 \quad\quad\hbox{and} \quad\quad \lim_{r\rightarrow r^{*}} \om'(r) \,=\, -1\,.
$$
In particular, we have that $\kappa=1$ and $(M^{n},g)$ is rotationally symmetric (see~\cite[Lemma 9.114]{besse}). Next we claim that, $t_{*}=-\infty$ and $t^{*}=+\infty$. In fact
$$
t_{*} \,=\, t(r_{0})-\lim_{r\rightarrow r_{*}} \int_{r}^{r_{0}} \frac{ds}{\om(s)} 
\quad\quad\hbox{and}\quad\quad
t^{*} \,=\, t(r_{0})+\lim_{r\rightarrow r^{*}} \int_{r_{0}}^{r} \frac{ds}{\om(s)} \,,
$$
for every fixed $r_{0}\in(r_{*},r^{*})$. Since $\om(s)\rightarrow 0$ and $\om'(s)\rightarrow 1$ as $s\rightarrow r_{*}$, we have that $\om(s)\simeq s-r_{*}$. Analogously $\om(s) \simeq r^{*}-s$, as $s\rightarrow r^{*}$. In particular, the two integrals diverge and the claim follows. Since $M^{n}$ is compact, the limits of $f'(r)$ as $r$ tends to $r_{*},r^{*}$ exist and are zero. Thus, the solution $(x,y,\om)$ converge to the equilibria $P=(1,0,0)$ and $Q=(-1,0,0)$ as $t$ tends to $-\infty$ and $+\infty$ respectively.

We pass now to consider complete, noncompact, gradient $\rho$--Einstein solitons which are warped product with canonical fibers. If $R_{rr}\equiv 0$, we have that the only admissible solutions are the flat $\RR^{n}$ or cylinders with canonical fibers and the interval of definition of $r$ is either the half straight line or the entire line respectively. More precisely we have the following classification.

\begin{teo} 
\label{teo-warped-cyl}
Let $(M^{n},g)$, $n\geq 3$, be a complete, noncompact, gradient $\rho$--Einstein soliton which is a warped product with canonical fibers as in~\eqref{eqwarp}. If $R_{rr}\equiv 0$, then $(M^{n},g)$ is either homothetic to the round cylinder $\RR\times\SS^{n-1}$, or to the hyperbolic cylinder $\RR\times\mathbb{H}^{n-1}$ or to the flat $\RR^{n}$.
\end{teo}
\begin{proof} We observe that, as a function of $r$, $\om$ is given by $\om(r)=\om_{0}+x_{0}(r-r_{0})$. Since $(M^{n},g)$ is assumed to be smooth and complete, the only admissible value of $x_{0}$ are $0,1$ and $-1$. 

{\bf Case 1.} When $x_{0}=0$, we have that $(r_{*},r^{*})=(-\infty,+\infty)$ and $\omega \equiv \om_0$. Depending on the sign of $1-(n-1)\rho$, one has the following cases.
\begin{itemize} 
\item $\rho <1/(n-1)$. In this case either we have $\lambda=0$, $\kappa=0$, no restrictions on $\om_0$ and the soliton is trivial; or $\lambda >0$, $\kappa =1$, $\om_0^2 = (n-2)(1-(n-1)\rho) / \l$ and $f(r) = \frac{\lambda }{2(1-(n-1)\rho)} {r^2} + a_0 r + b_0$, for some constants $a_0, b_0 \in \RR$; or $\lambda <0$, $\kappa =-1$, $\om_0^2 = - (n-2)(1-(n-1)\rho) / \l$ and $f(r) = \frac{\lambda }{2(1-(n-1)\rho)} {r^2} + a_0 r + b_0$, for some constants $a_0, b_0 \in \RR$.
\item $\rho = 1/(n-1)$. In this case we have $\lambda =0$, no restrictions on $\kappa$, no restrictions on $\omega_0$ and 
$f(r) = \frac{(n-2) \kappa }{2\omega_0^2} {r^2} + a_0 r + b_0$, for some constants $a_0, b_0 \in \RR$.
\item $\rho > 1/(n-1)$. In this case either we have $\lambda=0$, $\kappa=0$, no restrictions on $\om_0$ and the soliton is trivial; or $\lambda >0$, $\kappa =-1$, $\om_0^2 = - (n-2)(1-(n-1)\rho) / \l$ and $f(r) = \frac{\lambda }{2(1-(n-1)\rho)} {r^2} + a_0 r + b_0$, for some constants $a_0, b_0 \in \RR$; or $\lambda <0$, $\kappa =1$, $\om_0^2 =  (n-2)(1-(n-1)\rho) / \l$ and $f(r) = \frac{\lambda }{2(1-(n-1)\rho)} {r^2} + a_0 r + b_0$, for some constants $a_0, b_0 \in \RR$.
\end{itemize}

{\bf Case 2.} When $x_{0}=1$, we have that $(r_{*},r^{*})=(r_{0}-\om_{0},+\infty)$, $\om(r)=\om_{0}+r-r_{0}$,  and $\kappa=1$ (see~\cite[Lemma 9.114]{besse}). In particular, the metric is rotationally symmetric and {\em flat}, more precisely, $(M^n, g)$ is isometric to $(\RR^n, g^{\RR^n})$. Moreover, we have no restrictions on both $\lambda$ and $\om_0$ and $f(r)= \frac{\l}{2} r^2 + a_0 r + b_0$, for some constants $a_0, b_0 \in \RR$.

{\bf Case 3.} When $x_{0}=-1$, we have $(r_{*},r^{*})=(-\infty,\om_{0}-r_{0})$, $\om(r)=\om_{0}-r+r_{0}$  and $\kappa=1$. In particular, the metric is rotationally symmetric and {\em flat}, more precisely, $(M^n, g)$ is isometric to $(\RR^n, g^{\RR^n})$. Moreover, we have no restrictions on both $\lambda$ and $\om_0$ and $f(r)= \frac{\l}{2} r^2 + a_0 r + b_0$, for some constants $a_0, b_0 \in \RR$. This completes the proof.
\end{proof}

We pass now to consider complete, noncompact, gradient $\rho$--Einstein solitons which are warped product with canonical fibers for which $R_{rr}>0$ for every $r\in(r_{*},r^{*})$. First of all, we observe that the maximal interval of definition $(r_{*},r^{*})$ cannot coincide with $\RR$, since $\om$ is positive and strictly concave. Without loss of generality we can assume $-\infty<r_{*}$ and $r^{*} = + \infty$, since $M^n$ is noncompact (the same considerations will apply to the case $r^{*}<+\infty$ and $-\infty = r_{*}$). By smoothness of $(M^{n},g)$, we have that 
$$
\lim_{r\rightarrow r_{*}} \om(r) \,=\, 0 \quad\quad\hbox{and} \quad\quad \lim_{r\rightarrow r_{*}} \om'(r) \,=\, 1\,.
$$ 
In particular, we have that $\kappa=1$, thus, $(M^{n},g)$ is rotationally symmetric (see again~\cite[Lemma 9.114]{besse}) and diffeomorphic to $\RR^n$. We note incidentally that, from the point view of system~\eqref{system}, we are looking to solutions which `come out' from the equilibrium $P=(1,0,0)$. 

To proceed, we claim that $t_{*}=-\infty$ and $t^{*} = + \infty$. The first claim follows, reasoning as in the compact case. To prove that $t^* = + \infty$, we observe that, since $R_{rr}>0$, we have that $\dot{x}<0$ everywhere. Combining this with the fact that $\lim_{t \rightarrow -\infty} x(t) = 1$, we deduce that $\om' = x < 1$ everywhere. In particular, for every given $r_0 \in (r_{*}, + \infty)$, we have that $\om (r)< \om(r_0) + r-r_0$. Recalling that 
$$
t^{*} \,=\, t(r_{0})+\lim_{r\rightarrow +\infty} \int_{r_{0}}^{r} \frac{ds}{\om(s)} \,,
$$
and using the latter inequality for $\om$, it is immediate to see that the integral on the right hand side must diverge and the claim is proved. 

Another consequence of the fact that $\dot{x}<0$ everywhere is that $x$ must be strictly positive for all times. Indeed, if we assume that $x(t_0) = 0$ for some $t_0 \in \RR$, then, by the fact that $\dot{x}$ is strictly negative, we have that there exist $\ep>0$ and $t_{1}=t_{1}(\ep)\in\RR$ such that $x(t)<-\ep$ for every $t>t_{1}$. This would imply that $\om'(r)<-\ep$ for every $r>r_{1}=r(t_{1})$. Since $\om$ is defined for all $r\in[r_{*},+\infty)$, the latter condition would force $\om$ to become negative, which is geometrically unacceptable. Hence $0<x(t)<1$ for every $t$, in $\RR$.


Before starting the discussion of the steady soliton case, we observe that the sectional curvature of a rotationally symmetric metric as in~\eqref{eqwarp} are given by
$$
K_{rad} = -\frac{\,\om''}{\om}=-\frac{\dot{x}}{\om^2} \quad\quad\hbox{and}\quad\quad K_{sph} = \frac{1-(\om')^{2}}{\om^{2}} = \frac{1-x^{2}}{\om^{2}}\,,
$$
where $K_{rad}$ and $K_{spy}$ are the sectional curvatures of planes containing or perpendicular to the radial vector, respectively. Hence, a solution to the system~\eqref{systemdec} has positive sectional curvature if and only if $\dot{x}<0$ and $-1<x<1$, which is always the case, when $R_{rr}>0$. 

In the following theorem we classify solutions of the system with $R_{rr}>0$ for every $r\in(r_{*},r^{*})$ and $\lambda=0$.

\begin{teo} 
\label{teo-warp-steady}
If $\rho<1/2(n-1)$ or $\rho\geq 1/(n-1)$, $n\geq 3$, then, up to homotheties, there exists a unique complete, noncompact, gradient steady $\rho$--Einstein soliton which is a warped product with canonical fibers as in~\eqref{eqwarp} with $R_{rr}(r_{0})>0$ for some $r_{0}\in(r_{*},r^{*})$. This solution is rotationally symmetric and has positive sectional curvature.

If $1/2(n-1)\leq \rho<1/(n-1)$, then there are no complete, noncompact, gradient steady $\rho$--Einstein solitons which are warped products with canonical fibers as in~\eqref{eqwarp} with $R_{rr}(r_{0})>0$ for some $r_{0}\in(r_{*},r^{*})$.\end{teo}
\begin{proof} From the previous discussion, we have that: $\kappa=1$ (rotational symmetry), $(t_{*},t^{*})=\RR$, $(x(t),y(t),\om(t))\rightarrow P=(1,0,0)$ as $t\rightarrow -\infty$, $\dot{x}<0$ and $0<x<1$. Moreover, since $\lambda=0$, the system~\eqref{system}, reduces to the decoupled one
\begin{equation}\label{systemdec}
\begin{cases} 
(1-2m\rho)\,\dot{x} = (m-1)(1-m\rho)(1-x^{2}) -xy \\
(1-2m\rho) \,\dot{y} = -m(m-1)(1-(m+1)\rho)(1-x^{2})+(1+m-4m\rho)xy\\
\hspace{1.69cm}\dot{\om} = x\,\om\,.
\end{cases} 
\end{equation}

{\bf Case 1:} $\rho<1/2(n-1)=1/2m$. Since the system is decoupled, it is sufficient to consider the first two equations and exhibit an admissible trajectory $t \mapsto (x(t),y(t))$ defined for all $t \in \RR$ such that  $\lim_{t \rightarrow - \infty} (x(t),y(t))= (1,0)$. In the next, we are going to determine the support of such a trajectory. We start by observing that, since $\dot{x}<0$, one has
$$
(m-1)(1-m\rho)(1-x^{2})<xy\,.
$$
This implies at once that $\dot{y}>0$ for every $t\in\RR$. The fact that the solution we are looking for must come out of $(1,0)$ implies that, $y>0$. Taking advantage of these facts, we are going to consider $x$ as a function of $y$, with a small abuse of notations. It is now clear that the support of the admissible trajectory will coincide with the graph of a solution $x=x(y)$ of the ordinary differential equation
\begin{equation}\label{ode}
\frac{dx}{dy} \, = \, F(x(y), y) \, = \, \frac{(m-1)(1-m\rho)(1-x^{2}) -xy}{-m(m-1)(1-(m+1)\rho)(1-x^{2})+(1+m-4m\rho)xy} 
\end{equation}
defined for $y>0$ and such that $\lim_{y \rightarrow 0} x(y) = 1$. We prove the existence of such a solution, by taking the limit as $\ep \rightarrow 0$ of the family of solutions $x_{\ep}$, $\ep \in (0,1)$, to the following initial value problems
\begin{equation*}
\begin{cases} 
\,\,\displaystyle{\frac{dx_{\ep}}{dy}} \, = \, F(x_{\ep}(y), y) \, ,  \quad\quad {y \in \RR^{+}} \\
 x_{\ep}(0) = 1+ \ep\,.
\end{cases} 
\end{equation*}
We claim that for every $\ep \in (0,1)$, the function $x_{\ep}$ is defined for all $y\in \RR^{+}$, is monotonically decreasing and verifies the inequalities
$$
h(y) \, \leq \, x_{\ep} (y) \, \leq \, 1+\ep  \, ,
$$
for every $y \in \RR^{+}$, where the lower bound $h=h(y)$ is defined by
$$
h(y) \, = \, \frac{-y + \sqrt{y^{2} + 4(m-1)^{2}(1-m\rho)^{2}}}{2(m-1)(1-m\rho)} \,.
$$
We notice that $F(h(y), y) = 0$ and $dh/dy<0$ for every $y \in \RR^{+}$. Hence, every solution to equation~\eqref{ode} which is bigger than $1=h(0)$ at $y=0$ always stays bigger than $h(y)$ for every $y \in \RR^{+}$ where the solution exists. On the other hand, $(x,y)\mapsto F(x,y)$ is smooth and negative in the region $\{(x,y)\,\,|\,\,y>0 \hbox{ and } x\geq h(y)\}$. Combining these observations, the claim follows by standard ODE's theory. Moreover, it is easy to observe that, if $0<\ep_{1}<\ep_{2}<1$, then $x_{\ep_{1}}(y)<x_{\ep_{2}}(y)$ for  every $y \in \RR^{+}$. As a consequence of the claim, it is well defined the pointwise limit
$$
\overline{x}(y)\,=\,\lim_{\ep\rightarrow 0} x_{\ep}(y)\,,
$$ 
and $h(y)\leq \overline{x}(y)\leq 1$, for every $y \in \RR^{+}$. We want to prove that $\overline{x}(y)$ solves equation~\eqref{ode} in $\RR^{+}$. To do that we consider an exhaustion 
$[1/j,j]\subset \RR^{+}$, $j\in\NN$, and the associated family of compact sets $K_{j}=\{(x,y)\,\,|\,\,1/j\leq y\leq j \hbox{ and } h(y)\leq x \leq 2 \}\subset \RR^{2}$. Since, for every $j\in\NN$, $F\in C^{\infty}(K_{j})$ it is immediate to observe that $\Vert x_{\ep}\Vert_{C^{2}(K_{j})} \leq C_{j}$, for some positive constant $C_{j}$ independent of $\ep$. By Ascoli--Arzel\`a theorem, we have that $\overline{x}\in C^{1}(K_{j})$ for every $j\in\NN$. Hence, $\overline{x}(y)$ solves equation~\eqref{ode} in $\RR^{+}$ and for what we have seen, $\lim_{y \rightarrow 0} \overline{x}(y) = 1$.

To conclude,  we observe that since, $\dot{x}<0$ and $x>0$, then this solution has positive sectional curvature. 

{\bf Case 2:} $\rho\geq 1/(n-1)=1/m$. Reasoning as in the previous case, we are going to determine the support of an admissible trajectory $t \mapsto (x(t),y(t))$ defined for all $t \in \RR$ such that $\lim_{t \rightarrow - \infty} (x(t),y(t))= (1,0)$. We start by observing that, since $\dot{x}<0$, one has
$$
(m-1)(1-m\rho)(1-x^{2})>xy\,.
$$
This implies at once $\dot{y}<0$ for every $t\in\RR$. Hence $y<0$. As before, regarding $x$ as a function of $y$, we prove the existence of a solution $x=x(y)$ of the equation~\eqref{ode} on $y<0$ and such that $\lim_{y \rightarrow 0} x(y) = 1$. Setting $z=-y$, this is equivalent to prove the existence of a solution $x=x(z)$ to
\begin{equation}\label{odez}
\frac{dx}{dz} \, =\, G(x(z),z) \,=\, \frac{(m-1)(1-m\rho)(1-x^{2}) +xz}{m(m-1)(1-(m+1)\rho)(1-x^{2})+(1+m-4m\rho)xz}\,,
\end{equation}
defined on $z\in\RR^{+}$. We prove the existence of such a solution, by taking the limit as $\ep \rightarrow 0$ of the family of solutions $x_{\ep}$, $\ep \in (0,1)$, to the following initial value problems
\begin{equation*}
\begin{cases} 
\,\,\displaystyle{\frac{dx_{\ep}}{dz}} \, = \, G(x_{\ep}(z), z) \, ,  \quad\quad {z \in \RR^{+}} \\
 x_{\ep}(0) = 1 - \ep\,.
\end{cases} 
\end{equation*}
From now on we consider the case $\rho>1/m$. We claim that for every $\ep \in (0,1)$, the function $x_{\ep}$ is defined for all $z\in \RR^{+}$, and it is such that $x_{\ep}\leq 1$ for every $z\in\RR^{+}$. Moreover, there exists $z_{\ep}>0$ such that $k(z)\leq x_{\ep}(z)$ for every $z\geq z_{\ep}$, where the lower bound $k=k(z)$ is defined by
$$
k(z) \, = \, -\frac{z + \sqrt{z^{2} + 4(m-1)^{2}(1-m\rho)^{2}}}{2(m-1)(1-m\rho)} \,.
$$
We notice that $G(k(z), z) = 0$ and $dk/dz<0$ for every $z \in \RR^{+}$. Hence, if there exists $z_{\ep}$ such that $x(z_{\ep})=k(z_{\ep})$, then $x(z)\geq k(z)$ for every $z\geq z_{\ep}$. On the other hand, it is easy to observe that such a $z_{\ep}$ exists. In fact, if not, we would have a strictly increasing function, $x_{\ep}(z)$, which never crosses $k(z)$, but this contradicts the fact that $k$ tends to zero, as $z\rightarrow +\infty$. In particular, $x_{\ep}$ is strictly increasing before $z_{\ep}$, it has a maximum in $z_{\ep}$ and is strictly decreasing after $z_{\ep}$. Hence, $x_{\ep}(z)\leq x_{\ep}(z_{\ep})= k(z_{\ep})\leq 1$. Finally, as $(x,z)\mapsto G(x,z)$ is smooth in the region $(0,1)\times\RR^{+}$, the solution $x_{\ep}$ exists for every $z\in\RR^{+}$ and the claim follows. Moreover, as a consequence of standard ODE's comparison principle, it is easy to observe that, if $0<\ep_{1}<\ep_{2}<1$, then $x_{\ep_{1}}(z)>x_{\ep_{2}}(z)$ for  every $z \in \RR^{+}$. From what we have seen, it is now well defined the pointwise limit
$$
\hat{x}(z)\,=\,\lim_{\ep\rightarrow 0} x_{\ep}(z)\,,
$$ 
and $k(z)\leq \hat{x}(z)\leq 1$, for every $z \in \RR^{+}$, since $z_{\ep}\rightarrow 0$, as $\ep\rightarrow 0$. Adapting the arguments at the end of the previous case, it is immediate to prove the convergence in $C^{1}$--norm of the $x_{\ep}$'s to $\hat{x}$, which is now a solution of the equation~\eqref{odez} in $\RR^{+}$ with $\lim_{z \rightarrow 0} \hat{x}(z) = 1$.

The case $\rho=1/m$ can be treated in the same way and it is left to the reader. The main difference consists in the definition of the function $k$, namely one has to set $k(z)=0$ if $z>0$ and $k(0)=1$.   

To conclude,  we observe that since, $\dot{x}<0$ and $x>0$, then this solution has positive sectional curvature. 

{\bf Case 3:} $\rho=1/2(n-1)=1/2m$. In this case there are no solutions to the system~\eqref{systemdec} with $R_{rr}>0$, since the general identity~\eqref{equ2} implies at once $R_{rr}=0$ everywhere.

{\bf Case 4:} $1/2m=1/2(n-1)<\rho<1/n=1/(m+1)$. We start by observing that, since $\dot{x}<0$ and $\lim_{t \rightarrow - \infty} (x(t),y(t))= (1,0)$, one has
$$
(m-1)(1-m\rho)(1-x^{2})>xy\,.
$$
Notice that, $y\equiv 0$ is not admissible, since it would imply $x\equiv 1$ which contradicts $\dot{x}<0$. First of all we observe that in the region $\{0<x\leq1\}\cap\{y<0\}$ we have that $\dot{y}>0$. Hence, $y$ must be positive 
since $\lim_{t \rightarrow - \infty} (x(t),y(t))= (1,0)$. This limit also implies that there exists $t_{0}$ such that $\dot{y}(t_{0})>0$. We claim that $\dot{y}(t)\geq 0$ for every $t>t_{0}$. In fact, if $t_{1}>t_{0}$ is such that $\dot{y}(t_{1})=0$, then it is immediate to verify that the tangent vector of our trajectory at $t_{1}$, namely $(\dot{x}(t_{1}),0)$, is pointing inside the set $\{\dot{y}\geq 0\}$. Hence, $y>0$ and $\dot{y}\geq 0$ for every $t\geq t_{0}$. We claim that $x(t)$ cannot stay strictly positive for all times. In fact, if this happen the only possibility is that $\dot{x}\rightarrow 0$, $\dot{y}\rightarrow 0$, $x\rightarrow 0$ and $y\rightarrow +\infty$, as $t\rightarrow +\infty$. System~\eqref{systemdec} implies that the quantity $xy$ would tend to both $(m-1)(1-m\rho)$ and $-m(m-1)(1-(m+1)\rho)/(1+m-4m\rho)$, which is impossible since $1/2m<\rho < 1/(m+1)$. The claim is then proven. On the other hand, we have seen that if $R_{rr}>0$, then $0<x(t)<1$ for all $t\in \RR$. Thus we have reached a contradiction.


{\bf Case 5:} $\rho=1/n=1/(m+1)$. In this case one has $\dot{y}=-(m-1)xy$. We first observe that the solution $y\equiv 0$ is not admissible, since the manifold would be a noncompact space form with positive curvature. We claim that $y$ has a sign. In fact, if it is not the case, there exists $t_{0}$ such that $y(t_{0})=0$ and consequently $-\infty<t_{1}<t_{0}$ such that $\dot{y}(t_{1})=0$ and $y(t_{1})\neq 0$, which is impossible. On the other hand, $y$ and $\dot{y}$ must have the opposite sign and this contradicts the fact that $\lim_{t \rightarrow - \infty} (x(t),y(t))= (1,0)$. 

{\bf Case 6:} $1/(m+1)=1/n<\rho<1/(n-1)=1/m$. We start by observing that, since $\dot{x}<0$ and $\lim_{t \rightarrow - \infty} (x(t),y(t))= (1,0)$, one has
$$
(m-1)(1-m\rho)(1-x^{2})>xy\,.
$$
Again, $y\equiv 0$ is not admissible, since it would imply $x\equiv 1$ which contradicts $\dot{x}<0$. We claim that $y$ has a sign. In fact, if it is not the case, there exists $t_{0}$ such that $y(t_{0})=0$ and consequently $-\infty<t_{1}<t_{0}$ such that $\dot{y}(t_{1})=0$, which is impossible, since $\{\dot{y}\geq 0\}\cap\{\dot{x}<0\}\cap\{0<x<1\}=\emptyset$. Moreover, $\dot{x}<0$, implies at once that $y<0$ and $\dot{y}<0$ for every $t\in\RR$. In particular, one has
$$
\dot{x}\leq \frac{(m-1)(1-m\rho)}{(1-2m\rho)}(1-x^{2})\,,
$$
whenever $x\geq 0$. For a given $0<\ep<1$, we fix $t_{0}=t_{0}(\ep)$, such that $0<x(t_{0})=1-\ep$. This implies $x(t)<1-\ep$ and consequently 
$$
\dot{x}(t)< \frac{(m-1)(1-m\rho)\,\ep}{(1-2m\rho)}<0\,,
$$
for every $t>t_{0}$ and such that $x(t)\geq 0$. Since $t^{*}=+\infty$, we infer the existence of $t_{1}\in\RR$ such that $x(t_{1})=0$, which is unacceptable, since $0<x(t)<1$, for all $t \in \RR$.
\end{proof}

Combining Theorem~\ref{teo-warp-steady} with Theorem~\ref{teo-rot}, we obtain the following corollary.
\begin{cor}
\label{cor-rot} Up to homotheties, there is only one complete three--dimensional gradient steady $\rho$--Einstein soliton with $\rho<0$ or $\rho\geq 1/2$ and positive sectional curvature, namely the rotationally symmetric one constructed in Theorem~\ref{teo-warp-steady}.
\end{cor}
Combining Theorem~\ref{teo-warp-steady} with Theorem~\ref{teo-lcf},
we obtain the following corollary, which gives the classification of complete $n$--dimensional locally conformally flat gradient steady $\rho$--Einstein soliton with positive sectional curvature and $\rho\in \RR\setminus[0,1/2)$.
\begin{cor}
\label{cor-lcf} Up to homotheties, there is only one complete $n$--dimensional locally conformally flat gradient steady $\rho$--Einstein soliton with $\rho<0$ or $\rho \geq 1/2$ and positive sectional curvature, namely the rotationally symmetric one constructed in Theorem~\ref{teo-warp-steady}. 

%
\end{cor}
The last part of this section is devoted to the study of the asymptotic behavior of the gradient steady $\rho$--Einstein solitons constructed in Theorem~\ref{teo-warp-steady}, for $\rho < 1/2(n-1)$ and $\rho \geq 1/(n-1)$. To simplify the notations, we agree that, given two positive function $u(r)$ and $v(r)$, we have that $u = \mathcal{O} (v)$, for $r \rightarrow + \infty$, if and only if there exists two positive constants $A,B$ and $r_0$ such that, for every $r>r_0$,
$$
A \, v(r) < u(r)< B \, v(r) \, .
$$ 
We are now in the position to state the following proposition.
\begin{prop}
\label{prop-asymp}
Let $(M^n, g)$, $n\geq 3$, be the rotationally symmetric gradient steady $\rho$--Einstein soliton with positive sectional curvatures and normalized to have $R(O)=1$, constructed in Theorem~\ref{teo-warp-steady} for $\rho < 1/2(n-1)$ and $\rho \geq 1/(n-1)$. Then, with the notations of formula~\eqref{eqwarp}, we have that, as $r \rightarrow +\infty$,
$$
\omega(r) \, = \, \mathcal{O} \big(r^{\,\frac{1-(n-1)\rho}{2-3(n-1)\rho}} \big) \quad \quad \hbox{and} \quad \quad  |f(r)| \, = \, \mathcal{O} \big(r^{\,\frac{2-4(n-1)\rho}{2-3(n-1)\rho}} \big)  \,.$$
In particular, we have that $Vol_{g}(B_r(O))  \, = \, \mathcal{O} \big(r^{(n-1)\,\frac{1-(n-1)\rho}{2-3(n-1)\rho} +1 } \big)$, as $r \rightarrow + \infty$, where $B_r(O)$ is the ball of radius $r$ centered at the point $O$.
\end{prop}
\begin{proof}
We start by considering the case $\rho<1/2(n-1)$. First of all, we recall that if $t \mapsto (x(t), y(t), \omega(t))$ is the solution to system~\eqref{systemdec} under consideration, then we have $0<x <1$, $\dot x <0$ and $y,\dot y >0$ for every $t\in \RR$. Moreover it is easy to see that 
$$
\lim_{t\rightarrow +\infty} (x(t), y(t)) = (0, + \infty) \,.
$$ 
In fact, since $x$ is monotonically decreasing and bounded, it does have a limit, as $t\rightarrow + \infty$. If this limit is equal to some positive constant $a>0$, then, using the first equation in system~\eqref{systemdec} and the fact that $\dot x\rightarrow 0$, we would have that 
$$
y\rightarrow \frac{(n-2)(1-(n-1)\rho)(1-a^2)}{a} 
$$ 
and $\dot y \rightarrow 0$, as $t\rightarrow +\infty$. Using the second equation in system~\eqref{systemdec}, we would get 
$$
y \rightarrow \frac{(n-1)(n-2)(1-n\rho)(1-a^2)}{(n-4(n-1)\rho) a} \, ,
$$
as $t\rightarrow +\infty$. This would force $\rho = 1/2(n-1)$, which is excluded. 

The condition $\dot x < 0$ implies that
$$
xy > (n-2)(1-(n-1)\rho) (1-x^2) \,.
$$
Hence, $y \rightarrow + \infty$, as $t\rightarrow +\infty$.
By the system~\eqref{systemdec}, we infer that $xy \rightarrow (n-2)(1-(n-1)\rho)\neq 0$ and $\dot y \rightarrow (n-2) (1-2(n-1)\rho)$, as $t\rightarrow +\infty$. This implies that 
$$
\lim_{t \rightarrow + \infty} \frac{y(t)}{t} \, = \, (n-2)(1-2(n-1)\rho) \quad \quad \hbox{and} \quad \quad \lim_{t\rightarrow + \infty} t \, {x(t)} \, = \, \frac{1-(n-1)\rho}{1-2(n-1)\rho} \, .
$$
The equation $\dot \omega = x \omega$ implies that $\omega(t) = \mathcal{O} \big(\, t^{\frac{1-(n-1)\rho}{1-2(n-1)\rho}}\big)$, as $t \rightarrow + \infty$. Moreover, using the relationship $dr = (1/\omega) dt$, it is straightforward to conclude that 
$$
t \, = \, \mathcal{O} \big(\, r^{\frac{1-2(n-1)\rho}{2-3(n-1)\rho}}\big) \quad \quad \hbox{and} \quad \quad \omega \, = \, \mathcal{O} \big(\, r^{\frac{1-(n-1)\rho}{2-3(n-1)\rho}}\big) \, ,
$$
as $r \rightarrow + \infty$. The volume growth estimates contained in the statement are immediate consequences of the asymptotic behavior of $\omega$ described above. Finally, the fact that $y= -\omega f'$, implies at once the desired estimate for $|f(r)|$.

The proof is identical in the case $\rho > 1/(n-1)$ and it is left to the reader. Here, we only discuss the remaining  case, namely $\rho = 1/(n-1)$. Reasoning as before, we get that $xy \rightarrow 0$, as $t\rightarrow +\infty$, thus, we cannot go any further. However, in this case, the first equation in system~\eqref{systemdec} reads $\dot x = xy$, which implies $x= \mathcal{O}(\,e^{-\frac{n-2}{2} t^2})$, since 
$$
\lim_{t \rightarrow + \infty} \frac{y(t)}{t} \, = \, -(n-2) \, . 
$$
Using again the equation $\dot \omega = x \om $, we get 
$$
\om(t) \, = \, \mathcal{O} \bigg(   \exp \Big(  \int_{-\infty}^t e^{-\frac{n-2}{2} \, s^2} \, ds  \Big)  \,\,  \bigg) \,\, = \, \, \mathcal{O} (1) \, ,
$$
as $t\rightarrow +\infty$. In particular $\om(r) = \mathcal{O}(1)$ and $Vol_{g}(B_r(O))$, as $r\rightarrow + \infty$. Using the equation $f'' \om^2 - (n-3) \om \om'' + (n-2) (\om')^2 - (n-2) = 0$ and the fact that $\om' , \om'' \rightarrow 0$, as $r\rightarrow + \infty$, it is easy to deduce the estimates for the asymptotics of $|f(r)|$.    
\end{proof}
To conclude this section, we give some final comments on Theorem~\ref{teo-warp-steady} and Proposition~\ref{prop-asymp}. First, we notice that in the limit for $\rho \rightarrow 0$, the solutions provided in Theorem~\ref{teo-warp-steady} tend to the Bryant soliton metric, whose asymptotic behavior is determined by $\om(r) = \mathcal{O}(r^{1/2})$, $|f(r)| = \mathcal{O}(r)$ and $Vol_{g} (B_{r}(O)) = \mathcal{O}(r^{(n+1)/2})$, as $r \rightarrow + \infty$. 

An interesting feature of the solutions described in Theorem~\ref{teo-warp-steady} is that solutions corresponding to largely negative values of $\rho$ seem to be very close to solutions corresponding to largely positive values of $\rho$. In fact, the formal limit for $\rho \rightarrow \pm \infty$ of the asymptotic behavior is the same and it is given by $\om(r) = \mathcal{O}(r^{1/3})$, $|f(r)| = \mathcal{O}(r^{4/3})$ and $Vol_{g} (B_{r}(O)) = \mathcal{O}(r^{(n+2)/3})$, as $r \rightarrow + \infty$.

Another possible formal limit is the one for $\rho \rightarrow 1/2(n-1)$. In this case, the solutions provided by Theorem~\ref{teo-warp-steady} tend to Schouten solitons, which we are going to discuss in the next section. In particular, the formal limit of the asymptotic behavior is of Euclidean type and it is given by $\om(r) = \mathcal{O}(r)$, $|f(r)| = \mathcal{O}(1)$ and $Vol_{g} (B_{r}(O)) = \mathcal{O}(r^{n})$, as $r \rightarrow + \infty$. This perfectly agrees with the conclusions in Theorem~\ref{teo-ssteady} below.

Among all the solutions constructed in Theorem~\ref{teo-warp-steady}, probably the most significant ones correspond to the value $\rho = 1/(n-1)$. In this case, Proposition~\ref{prop-asymp} implies that $\om(r) = \mathcal{O}(1)$, $|f(r)| = \mathcal{O}(r^2)$ and $Vol_{g} (B_{r}(O)) = \mathcal{O}(r)$, as $r \rightarrow + \infty$. Hence, these solitons have linear volume growth and are asymptotically cylindrical. Moreover, we notice that in dimension $n=2$, the equation for a $1/(n-1)$--Einstein soliton reads
$$
\nabla^2 f = \frac{1}{2} R \,  g \, ,
$$
which, up to change the sign of $f$, coincides with the equation of two--dimensional gradient steady Ricci solitons. In this case, the only complete noncompact solution with positive curvature is the Hamilton's cigar~\cite{hamilton5}, also known as Witten's black hole. For these reasons, it is natural to consider the rotationally symmetric gradient steady $1/(n-1)$--Einstein solitons as the $n$--dimensional generalization of the Hamilton's cigar, hence, we will call them {\em cigar--type solitons}. In dimension $n=3$, it turns out that this solution is an Einstein soliton. Thus, we will refer to it as the {\em Einstein's cigar}. In this special situation, Corollary~\ref{cor-rot} may be rephrased in the following way.
\begin{cor}
\label{cor-e-cigar} Up to homotheties, the only complete three--dimensional gradient steady Einstein soliton with positive sectional curvature is the Einstein's cigar.
\end{cor}
For $n\geq4$, in the locally conformally flat case, we have the following corollary.
\begin{cor}
\label{cor-cigar} Up to homotheties, the only complete n--dimensional locally conformally flat gradient steady $1/(n-1)$--Einstein soliton with positive sectional curvature is the cigar--type soliton.
\end{cor}

\

\section{Gradient Schouten solitons}
\label{Schouten solitons}

In this section we classify $n$--dimensional steady and three--dimensional shrinking gradient Schouten solitons. We recall that a gradient  Schouten soliton is a Riemannian manifold $(M^{n},g)$, $n\geq 3$, satisfying
\begin{equation}\label{ssol}
\Ric + \nabla^{2}f = \frac{1}{2(n-1)} R\,g + \lambda g\,,
\end{equation}
for some smooth function $f$ and some constant $\lambda \in \RR$. We start by observing that ancient solutions to the Schouten flow 
\begin{equation}
\label{sflow}
\partial_t g \, = \, -2 \Big( Ric - \tfrac{1}{2(n-1)} R g \Big)\,,
\end{equation}
must have nonnegative scalar curvature. 
\begin{prop}
\label{sancient}
Let $\big( M^n, g(t) \big)$, $n \geq 3$, $t \in (-\infty, T)$, be a complete ancient solution to the Schouten flow~\eqref{sflow}. Then, $g(t)$ has nonnegative scalar curvature for every $t \in (-\infty, T)$.
\end{prop}
\begin{proof}
Using the general formula for the first variation of the scalar curvature (see~\cite[Theorem 1.174]{besse}) it is immediate to obtain the following evolution of $R$
\begin{equation}
\partial_t R \, = \, 2|Ric|^2 - \frac{1}{n-1} R^2 \, .
\end{equation}
In particular, for every $p \in M^n$, one has that 
$$
\partial_t R \, \geq \, \frac{n-2}{n(n-1)} R^2 \, .
$$
Thus, at every given point $p$, the scalar curvature is a nondecreasing in $t$. We fix now $t \in (-\infty, T)$ and we choose $t_0 \in (\infty, t)$. By the ODE comparison principle, one has that 
$$
R(t) \, \geq \, \frac{ n (n-1) R(t_0)}{n(n-1)-(n-2) R(t_0) (t-t_0)} \, .
$$
If $R(t_0) \geq 0$, then $R(t)$ is nonnegative, by monotonicity. Hence, we assume $R(t_0)$ to be strictly negative and we let $t_0$ tend to $- \infty$, obtaining $R(t) \geq 0$. Since both $p \in M^n$ and $t \in (\infty, T)$ were chosen arbitrarily, the proof is complete. 
\end{proof}
An immediate consequence of Proposition~\ref{sancient} is the following corollary.
\begin{cor}
\label{spsc}
Let $(M^n, g)$, $n\geq3$, be a complete shrinking or steady Schouten soliton. Then, $g$ has nonnegative scalar curvature. 
\end{cor}
We focus now our attention on the steady case and we prove the following theorem.
\begin{teo}
\label{teo-ssteady}
Every complete gradient steady Schouten soliton is trivial, hence Ricci flat.
\end{teo}
\begin{proof}
If the soliton is compact, the statement follows from Theorem~\ref{teo-cptesol}, case (i-bis). Thus, we assume $(M^n, g)$, $n\geq 3$, to be complete and noncompact. 

We observe that if $R\equiv 0$ everywhere, then identity~\eqref{equ3}, which in this case reads
\begin{equation}
\label{equ4}
0 = \langle \nabla R,\nabla f\rangle + \frac{1}{n-1} R^{2} - 2|\Ric|^{2}   \,,
\end{equation}
implies at once $\Ric \equiv 0$. Hence, the soliton is trivial.

To prove the statement, we suppose by contradiction that $R_{|p}>0$, for some $p \in M^n$. We claim that the connected component $\Sigma_0$ of the level set of $f$ through $p$ is regular. This follows by observing that at $p$ one has
$$
\langle \nabla R,\nabla f\rangle_{|p} \, \geq \, \frac{n-2}{n(n-1)} R^2_{|p} \, > \, 0 \, ,
$$
where we used the standard inequality $|Ric|^2 \geq R^2/n$. In particular, we have that $|\nabla f| \neq 0$ in $p$. We let $r$ be the signed distance to $\Sigma_0$, defined on a maximal interval $(r_*, r^*)$. By Theorem~\ref{teo-esolrect}, we have that $f$ only depends on $r$ an that $|\nabla f|$ only depends on $r$ as well (to be definite, we choose the sign of $r$ in such a way that $f' = |\nabla f|$). In particular, this implies that $\Sigma_0$ is regular and the claim is proved.

As a next step towards the contradiction, we are going to show that $f'(r)>0$, for $r>0$. First of all, we observe that, by identity~\eqref{equ2} in Theorem~\ref{teo-cptesol}, we have $Ric (\nabla f , \, \cdot \, ) = 0$. Thus, the $(r,r)$--component of equation~\eqref{ssol} reads
$$
f'' \, = \, \frac{1}{2(n-1)} R \,.
$$
Corollary~\ref{spsc} implies that $f''\geq 0$. Since $\Sigma_0$ has been assumed to be a regular level set, we have that $f' (0) >0$. We claim that $r^* = +\infty$. In fact, if it would not be the case, then we would have that $f' (r)\rightarrow 0$, as $r\rightarrow r^*$, which is clearly impossible. This implies that the signed distance is defined on all of $(r_*, +\infty)$ and that $f'$ always stays positive for all $r>0$. 

From~\eqref{equ4} and Proposition~\ref{prop-level} we have
$$
R' f' \, = \, 2|Ric|^2 - \frac{1}{n-1}R^2 \, \geq \, \frac{1}{n-1} R^2 \,\, ,
$$
where we used the inequality $|Ric|^2 \geq R^2/(n-1)$, which is a consequence of $Ric(\nabla f, \, \cdot \,) = 0$. For what we have seen and since $f'(r)>0$ for $r>0$, we infer that $R$ is nondecreasing and in particular it is strictly positive, for $r>0$. Dividing the above inequality by $R^2$ and integrating (by parts) between $0$ and $r>0$, gives 
$$
\frac{f'}{R}(r) \, \leq \, \frac{f'}{R} (0) \, - \, \frac{r}{2(n-1)} \, .
$$
Up to choose $r$ sufficiently large, the right hand side becomes negative, which is a contradiction.
%
\end{proof}

We pass now to consider the case of complete shrinking Schouten solitons. We restrict ourselves to the three--dimensional case and we prove the following Theorem.
\begin{teo}
\label{teo-sshrinking}
Let $(M^3, g)$ be a complete three dimensional gradient shrinking Schouten soliton. Then, it is isometric to a finite quotient of either $\SS^3$, or $\RR^3$ or $\RR \times \SS^2$.
\end{teo}
\begin{proof}
Up to lift the metric to the universal cover, we can assume, without loss of generality, that $(M^{3},g)$ is simply connected.

If the soliton is compact, it follows from Theorem~\ref{teo-cptesol}, case (i-bis) that $(M^3, g)$ is isometric to $\SS^3$. Thus, we assume $(M^3, g)$ to be complete and noncompact. 

From now on, we also assume that $f$ is nonconstant, otherwise the soliton is trivial and, again, it is isometric to $\SS^3$. In particular, there exists a regular connected component $\Sigma_0$ of some level set of $f$ and we let $r$ be the signed distance to $\Sigma_0$, defined on a maximal interval $(r_*, r^*)$. Here maximality has to be understood in the sense that $|\nabla f| \neq 0$ in $(r_*, r^*) \times \Sigma_0$ and eventually annihilates at the boundary. To be definite, we choose the sign of $r$ in such a way that $f' = |\nabla f|$.  By Theorem~\ref{teo-esolrect} and Proposition~\ref{prop-level}, we also have that $f, |\nabla f|, R$ and $R^{\Sigma}$, which is the scalar curvature induced on the level sets of $f$, only depend on $r$. 

We observe that, by identity~\eqref{equ2} in Theorem~\ref{teo-cptesol}, we have $Ric (\nabla f , \, \cdot \, ) = 0$. Thus, the $(r,r)$--component of equation~\eqref{ssol} reads
$$
f'' \, = \, \frac{1}{4} R + \l\,.
$$
Corollary~\ref{spsc} implies that $f''\geq \l>0 $. Since $\Sigma_0$ has been assumed to be a regular connected component of a level set, we have that $f' (0) >0$. We claim that $r^* = +\infty$. In fact, if it would not be the case, then we would have that $f' (r)\rightarrow 0$, as $r\rightarrow r^*$, which is clearly impossible. This implies that the signed distance is defined on all of $(r_*, +\infty)$ and that $f'$ always stays positive for all $r>0$. 

We claim that 
\begin{eqnarray}
\label{claim}
R^\Sigma_{|p} >0 \, , & & \hbox{for some $p \in (r_*, +\infty) \times \Sigma_0 $}\, .
\end{eqnarray}
We notice that the statement of the theorem is a consequence of this claim. In fact, if $r_0 = dist(p, \Sigma_0)$ and we denote by $\Sigma(r_0)$ the (regular) level set of $f$ through $p$, we have that there exists a maximal tubular neighborhood $U$ of $\Sigma({r_0})$, where the scalar curvature induced on the level sets of $f$ remains strictly positive. Since the manifold is three--dimensional, it follows from the two-dimensional Uniformization Theorem (applied to the level sets of $f$) that in $U$ the metric $g$ is a warped product with canonical fibers of positive constant curvature, that is, with the notations of Section~\ref{sect-warped}, $g = dr\otimes dr + \omega^2(r) g_{\SS^2} $. Moreover, we have that $\omega''=0$ in $U$, since $Ric(\nabla f, \, \cdot\,) = 0$. This implies that $\omega(r) = a (r-r_0) +b$, for some constant $a,b \in \RR$, with $a\geq 0$ and $b>0$. As a consequence, we have that, if $a= 0$, then $g$ is locally isometric to $\RR \times \SS^2$ and $f(r) = \l (r-r_0)^2 + c(r-r_0) +d$. On the other hand, if $a\neq 0$, then $g$ is locally isometric to $\RR^3$ and $f=\frac{\l}{2} (r-r_0)^2 + \frac{\l b}{a} (r-r_0) + e $, for some constants $c,d,e \in \RR$. 

Using the Gauss equation, we get $R^{\Sigma} (r) = 2/b^2$ or $R^{\Sigma}(r) = 2a/(a(r-r_0) +b)^2$, respectively. By the maximality of $U$ and since everything is smooth, we deduce that in the first case $r_*=-\infty$ and $(M^3,g)$ is globally isometric to $\RR \times \SS^2$, whereas, in the second case $-\infty<r_*$ and $(M^3,g)$ is globally isometric to $\RR^3$.

To prove the claim~\eqref{claim}, we argue by contradiction and we suppose that $R^{\Sigma}(r) \leq 0$, for every $r\in (r_* , +\infty)$. Up to add a constant, we can assume that $f(0)=0$. We are going to prove that
\begin{equation}
\label{positive}
0 \,< \,  \int_{ \Omega } R^{\Sigma} \, |\nabla f|^2 \,  e^{-f} dV_g  \, < \, +\infty \, ,
\end{equation}
where $\Omega = \{ 1< f \} \cap (0, +\infty) \times \Sigma_0$, which is clearly a contradiction. As a first step, we want to obtain a suitable expression for the integrand. Using the computations in Proposition~\ref{prop-level}, one has that
$$
R^{\Sigma} \, = \, R + H^2-|h|^2 \, ,
$$
where $ |\nabla f| \, h_{ij} = R_{ij} - (\frac{1}{4} R + \l) g_{ij}$ and $|\nabla f| \, H = \frac{1}{2} R - 2 \l$. Substituting the last two expression, one gets
\begin{equation}
\label{equ5}
R^{\Sigma} \, |\nabla f|^2  \, = \, R \, |\nabla f|^2 - |Ric|^2 + \frac{5}{8} R^2 - \l R + 2 \l^2 \, .
\end{equation}

For the rest of the argument we will assume that all integrals involved are finite and the integration by parts can be performed, which we shall justify after we complete the formal argument. To proceed, we recall the identities~\eqref{equ1} and \eqref{equ3}, which in the present situation read
\begin{eqnarray}\label{equ6}
\langle \nabla R,\nabla f\rangle &  = &  2 \, |\Ric|^{2}  -\, \frac{1}{2} R^2 - 2 \lambda R\,, \\ 
\label{equ7}
\Delta f & = & -\frac{1}{4} R + 3 \l \,.
\end{eqnarray}
A first formal integration by parts using equation~\eqref{equ7} gives
\begin{eqnarray*}
\int_{ \Omega} \langle \nabla R,\nabla f\rangle \, e^{-f} \, dV_g  & = & \int_{ \Omega} \big(  R\, |\nabla f|^2  -  R \Delta f     \big) \, e^{-f} \, dV_{g} \, - \,  \int_{\partial \Omega}R \, |\nabla f| e^{-f} d\sigma_{g} \, \\
& = &  \int_{\Omega} \big(  R\, |\nabla f|^2 +\frac{1}{4} R^2 - 3\l R\,  \big) \, e^{-f} \, dV_{g} \, - \, \int_{\partial \Omega}R \, |\nabla f| e^{-f} d\sigma_{g} \, ,
\end{eqnarray*}
where $d\sigma_g$ is the area element induced by $g$ on the boundary of $\Omega$. Now, using equation~\eqref{equ6}, we get
\begin{equation*}
\int_{ \Omega}   R\, |\nabla f|^2   e^{-f} \, dV_{g}  \, = \,  \int_{ \Omega} \big( \, 2|Ric|^2 - \frac{3}{4} R^2 + \l R \,  \big) \, e^{-f} \, dV_{g} \,
 + \, \int_{\partial \Omega}R \, |\nabla f| e^{-f} d\sigma_{g} \, .
\end{equation*}
Taking advantage of the last expression, we integrate equation~\eqref{equ5}, obtaining
\begin{eqnarray*}
\int_{ \Omega}   R^{\Sigma}\, |\nabla f|^2   e^{-f} \, dV_{g}  & = &  \int_{ \Omega} \big( \, |Ric|^2 - \frac{1}{8} R^2 +2\l^2 \,  \big) \, e^{-f} \, dV_{g} \,
 + \int_{\partial \Omega}R \, |\nabla f| e^{-f} d\sigma_{g}   \\
 & \geq & \int_{ \Omega} \Big( \,  \frac{5}{24} R^2 +2\l^2 \,  \Big) \, e^{-f} \, dV_{g} \,\, > \,\, 0 \, .
\end{eqnarray*}
This proves the first inequality in~\eqref{positive}, concluding the formal argument. 

To complete the proof, we need to justify the integrations by parts, showing that all the integrals involved are finite. This will be done in several steps.

{\bf Step 1.} As a first step, we show that the scalar curvature $R$ is necessarily bounded in $\Omega$. From equation~\eqref{equ6}, we have that 
$$
R' f' \, = \, 2|Ric|^2 - \frac{1}{2}R^2  - 2\l R\, \geq \, \frac{1}{2} R^2 -2 \l R \,\, ,
$$
where we used the inequality $|Ric|^2 \geq R^2/2$, which  follows from the fact that $Ric(\nabla f, \, \cdot \,) =0$. If $R$ would not be bounded, then, for a fixed real number $0<\delta < 1/4$, it would exists a suitable distance $r_\delta>0$ such that $R'f' \geq (\frac{1}{2} - \delta) R^2$ and $R > 8 \lambda$, for every $r>r_\delta$. For what we have seen, since $f'(r)>0$ whenever $r>0$, we infer that $R$ is nondecreasing and in particular it is strictly positive, for $r>r_\delta$. The same argument used in Theorem~\ref{teo-ssteady} shows that 
$$
\frac{f'}{R}(r) \, \leq \, \frac{f'}{R} (r_\delta) \, - \, \Big( \frac{1}{8} -\delta \Big) (r-r_\delta) \, .
$$
Up to choose $r$ sufficiently large, the right hand side becomes negative, which is a contradiction. Hence, we have proved that $R$ must be bounded in $ \Omega$. Implicitly, this argument shows that, in $\Omega$, the scalar curvature $R$ cannot be larger than $8 \l$.

{\bf Step 2.} The second step amounts to prove that for every $p_0 \in \Sigma_0$, there exist positive constants $c_1, c_2$ and $C$ such that for every $p \in \Omega$
\begin{equation}
\label{f-growth}
\frac{1}{C} (d(p) - c_1)^2 \, \leq \,  f(p)  \,\leq \, C(d(p) + c_2)^2 \, ,
\end{equation}
where $d(p) = dist(p, p_0)$. In order to prove the upper bound, we start by observing that, up to choose a sufficiently large constant $a>0$, the quantity $a f (r)- |\nabla f|^2(r)$ is monotonically increasing in $r$, for $r>0$. In fact, from the Schouten soliton equation~\eqref{ssol} we have that
\begin{eqnarray*}
\langle  \nabla (af-|\nabla f|^2) , \nabla f  \rangle & = &  a |\nabla f|^2 - 2 \nabla^2 f (\nabla f, \nabla f) \\
& = & \big(a -2\l - \frac{1}{2} R \big) \, |\nabla f|^2 \, ,
\end{eqnarray*}
which is clearly positive, provided the constant $a>0$ is large enough, since by Step 1 the scalar curvature $R$ is bounded in $\Omega$. On the other hand, it is easy to observe that
\begin{eqnarray*}
\langle  \nabla (|\nabla f|^2-2\l f) , \nabla f  \rangle  \,=\, \frac{1}{2}R |\nabla f|^2 \, > \, 0 \,.
\end{eqnarray*}
Putting these two latter estimates together, we obtain that, for every $r>0$, 
\begin{equation}
\label{grad-ineq}
2\l f(r) + (f'(0))^{2} \,\leq \, |\nabla f|^2(r) \, \leq \, a f(r) +  (f'(0))^2 \, .
\end{equation}
These inequalities play the role of Hamilton's  identity for gradient Ricci solitons (see~\cite{hamilton9}), which turns out to be fundamental in proving growth estimates on potential function. In particular, inequalities~\eqref{grad-ineq} imply that $|\nabla \sqrt{f}|$ is bounded and $\sqrt{f}$ is Lipschitz in $\Omega$. This proves the upper bound in~\eqref{f-growth}.  

To prove the other estimate, we can adapt step by step the proof of the lower bound for the potential function of gradient shrinking  Ricci solitons presented in~\cite[Proposition 2.1]{caozhou}. In fact, the computations in the Schouten soliton case differs from the Ricci soliton one only by some correction terms, involving the scalar curvature. In particular, using the fact that $R \geq 0$, identity $(2.7)$ in~\cite[Proposition 2.1]{caozhou} can be replaced in our case by the matrix inequality
$$
\nabla^2 f \, \geq \, \l g - Ric \, .
$$
All the other estimates in~\cite[Proposition 2.1]{caozhou} remain true till inequality $(2.9)$.

The other key ingredient in the proof by Cao and Zhou is their inequality $(2.10)$, which in the present situation can be replaced by 
$$
\max_{s_0 -1 \leq s \leq s_0} |\nabla_{\dot{\gamma}(s)} f| (\gamma(s)) \, \leq \, a_1 \sqrt{f(\gamma(s_0))} + a_2 \,,
$$
where $\gamma$ is a geodesic starting from $p_0$ and supported in $\{ r>0\}$, and $a_1$ and $a_2$ are suitable positive constants, eventually depending on the constant $a>0$.


{\bf Step 3.} We prove now a volume growth estimate for the sub--level sets of $f$. More precisely, there exists a positive constant $A$, such that
\begin{equation}
\label{vgrowth}
Vol_g (\{ 0<f <s\} \cap \Omega) \, \leq \, A  \, s^{3/2} \, . 
\end{equation}

In the spirit of~\cite{caozhou}, we define in the set $\{r \geq 0\}$ the function $u = 2 \sqrt{f}$ and we set $D(s) = \{2<u<s\} \subset \Omega$. First of all, we notice that an immediate consequence of the double inequality~\eqref{f-growth} proved in Step 2 is the fact that the sets $D(s)$ are compact for every $s>2$. Setting $V(s) = Vol_g(D(s))$, by the co--area formula we have that
$$
V(s) \, = \, \int_2^s dt \int_{\partial D(t)} \frac{1}{|\nabla u|}  \,  dS_g(t) \, ,
$$ 
where $dS_g(\cdot)$ is the area element induced by $g$ on $\partial   D(\cdot)$. We also notice that for every $s>0$, the boundary of $D(s)$ is given by the disjoint union of $\partial^+ D(s) = \{u=s \}$ and $\partial^- D(s)=\{u=2 \}$. Hence, tacking advantage of the rectifiability of $f$, we easily compute
$$
V'(s) \, = \, \int_{\partial D(s)} \frac{1}{|\nabla u|} dS_g (s)\, = \, \frac{ s \, |\partial^+D(s)|}{ 2\, |\nabla f|_{| \partial^+D(s)}}   + \frac{   |\partial^-D(s)|}{  |\nabla f|_{| \partial^-D(s)}}
 \,.
$$ 
Integrating equation~\eqref{equ1} on $D(s)$, we get
\begin{eqnarray*}
3 \l V(s) - \frac{1}{4}\int_{D(s)}R\,dV_{g} &=&  \int_{D(s)}\Delta f\,dV_{g} \\
&=&  \int_{\partial^+D(s)} |\nabla f|\,dS_{g}(s)  \, - \, \int_{\partial^-D(s)} |\nabla f|\,dS_{g}(s)    \\
&=&  |\nabla f|_{\partial^+D(s)} \,  |\partial^+D(s)| \, - \, 
|\nabla f|_{\partial^-D(s)} \,  |\partial^-D(s)|
 \,.
\end{eqnarray*}
Hence, using the formula for $V'(s)$ and the fact that $R$ is nonnegative, we obtain the inequality
$$
3 \l V(s) \, \geq \, \frac{2}{s} \, |\nabla f|^2_{|\partial^+D(s)} \, V'(s) \, - \, \frac{2}{s} \, \frac{\partial^-D(s)}{|\nabla f|_{|\partial^-D(s)}} \, |\nabla f|^2_{\partial^+D(s)} \, - \, |\partial^-D(s)| \, |\nabla f|_{\partial^-D(s)} \, 
$$
Now, we observe that in the present situation, estimates~\eqref{grad-ineq} implies that
$$
\frac{\l}{2} s^{2} \, \leq \, |\nabla f|^2_{|\partial^+D(s)} \, \leq \, \frac{a}{4}s^2  \,+ \, |f'(0)|^2
$$
Combining the last two inequality, we obtain
$$
3 \lambda V(s) \, \geq \, \l s V'(s) - A_1 s - A_2 -\frac{A_3}{s}\, , 
$$
where we set $A_1= a |\partial^-D(s)| / 2 |\nabla f|_{|\partial^-D(s)}$, $A_2 = |\partial^-D(s)| \, |\nabla f|_{|\partial^-D(s)}$ and $A_3 = 2|f'(0)|$. Thus, we have proved that for large enough $s$ 
$$
V'(s) \, \leq \, 3  \, \Big( \frac{V(s)}{s} + \frac{A_1}{\lambda} \Big)  \,.
$$
Setting $W(s) =  (V(s)/s + A_1/\l )$, we have $(W'/W)(s) \leq 3/s$. Integrating this inequality and using the definition of $W$, we get $V(s) \,\leq \, B \, s^{3}$, for some positive constant $B$. Finally, going back to the definition of $V$, we obtain the desired estimate~\eqref{vgrowth}.

Using {\bf Step 1}, {\bf Step 2} and {\bf Step 3} it is now an easy exercise to check that all the integrations by parts performed in the formal argument are justified.
\end{proof}

\

\section{Concluding remarks and open questions}
\label{sect-open}

To conclude, we present a short list of comments and open questions, which could be the subject of further investigation.
   
\begin{enumerate}

\smallskip
\item 
In Theorem~\ref{teo-cptesol} we have seen some triviality results for compact gradient $\rho$--Einstein solitons. It would be interesting to investigate whether in cases $(i)$ and $(ii)$, one could get the same conclusion as in cases $(i-bis)$ and $(iii)$. For example, in analogy with gradient Ricci solitons, we expect that the only compact three--dimensional gradient shrinking $\rho$--Einstein soliton with $\rho < 1/4$ is a quotient of the round sphere $\SS^3$. On the other hand, it would be interesting to construct examples of compact nontrivial gradient $\rho$--Einstein solitons with $\rho < 1/2(n-1)$ in dimension $n\geq 4$. In the case of Ricci solitons, this has been done  by several authors (see~\cite{cao1},~\cite{koiso1} and~\cite{wangzhu}).

\smallskip
\item
In Corollary~\ref{cor-rot}, we have seen that up to homotheties, there exists only one three--dimensional gradient steady $\rho$--Einstein soliton with positive sectional curvature, provided $\rho <0$ or $\rho\geq 1/2$. In reason of Theorems~\ref{teo-warp-steady} and~\ref{teo-lcf} we expect that the same conclusion holds also for $0 < \rho<1/4$, without any further assumption. 
We recall that in Theorem~\ref{teo-ssteady} we have shown that every complete gradient steady Schouten soliton $(\rho = 1/4)$ is trivial. Also notice that for ``$\rho=0$'', this is the Perelman's claim, mentioned in the introduction.

\smallskip

\item
In Corollary~\ref{cor-lcf}, we have seen that, up to homotheties, there exists only one $n$--dimensional locally conformally flat gradient steady $\rho$--Einstein soliton with positive sectional curvature, provided $\rho < 0$ or $\rho\geq 1/2$ and $n\geq 3$. We recall that in Theorem~\ref{teo-ssteady} we have shown that every complete gradient steady Schouten soliton $(\rho = 1/2(n-1))$ is trivial. 
Also notice that for ``$\rho=0$'', the existence of a unique locally conformally flat gradient steady Ricci soliton was already known (see~\cite{caochen} and~\cite{mancat1}). Moreover, in this case the assumption about locally conformally flatness can be replaced with weaker conditions such as the harmonicity of the Weyl tensor or even the Bach flatness~\cite{caochecatmanmaz}. We expect that the same techniques would apply to the case of gradient steady $\rho$--Einstein solitons, with $\rho$ in the same ranges as in Theorem~\ref{teo-warp-steady}.

\smallskip
\item
It would be important to further exploit all the geometric consequences of the rectifiability. A possible direction of investigation is to prove a rigidity results for noncompact gradient shrinking $\rho$--Einstein solitons, in analogy with the case of rectifiable gradient Ricci solitons, studied in~\cite{pw2}. More precisely, we expect 
that for $\rho \leq 1/2(n-1)$ every noncompact gradient shrinking $\rho$--Einstein solitons with nonnegative (radial) sectional curvatures is rigid, namely isometric to a quotient of a direct product of the type $\RR^k \times N^{n-k}$, where $N$ is a $(n-k)$--dimensional compact Einstein manifold, for some $1\leq k \leq n$.

\smallskip
\item
Concerning the analysis of the complete noncompact rotationally symmetric gradient $\rho$--Einstein soliton, it would be interesting to prove the analogous of Theorem~\ref{teo-warp-steady} for shrinking and expanding solitons with positive sectional curvature. 

\end{enumerate}

\

\

\noindent {\bf Added note.} Shortly after this manuscript appeared, S. Brendle posted the article~\cite{brendle3} on the ArXiv, where Perelman's claim is proved.

\

\bigskip

\bibliographystyle{amsplain}
\bibliography{eRecto}

\providecommand{\bysame}{\leavevmode\hbox to3em{\hrulefill}\thinspace}
\providecommand{\MR}{\relax\ifhmode\unskip\space\fi MR }
\providecommand{\MRhref}[2]{%
  \href{http://www.ams.org/mathscinet-getitem?mr=#1}{#2}
}
\providecommand{\href}[2]{#2}
\begin{thebibliography}{10}

\bibitem{besse}
A.~L. Besse, \emph{Einstein manifolds}, Springer--Verlag, Berlin, 2008.

\bibitem{bishoponeill}
R.~L. Bishop and B.~O'Neill, \emph{Manifolds of negative curvature}, Trans.
  Amer. Math. Soc. \textbf{145} (1969), 1--49.

\bibitem{brendle3}
S.~Brendle, \emph{Rotational symmetry of self-similar solutions to the {R}icci
  flow}, ArXiv Preprint Server -- http://arxiv.org, 2012.

\bibitem{cao1}
H.-D. Cao, \emph{Existence of gradient {K}\"ahler-{R}icci solitons}, Elliptic
  and parabolic methods in geometry (Minneapolis, MN, 1994), A K Peters,
  Wellesley, MA, 1996, pp.~1--16.

\bibitem{cao3}
\bysame, \emph{Geometry of {R}icci solitons}, Chinese Ann. Math. Ser. B
  \textbf{27} (2006), no.~2, 121--142.

\bibitem{cao2}
\bysame, \emph{Recent progress on {R}icci solitons}, Adv. Lect. Math. (ALM)
  \textbf{11} (2010), 1--38.

\bibitem{caochenzhu}
H.-D. Cao, B.-L. Chen, and X.-P. Zhu, \emph{Recent developments on {H}amilton's
  {R}icci flow}, Surveys in differential geometry. {V}ol. {XII}. {G}eometric
  flows, vol.~12, Int. Press, Somerville, MA, 2008, pp.~47--112.

\bibitem{caochen}
H.-D. Cao and Q.~Chen, \emph{On locally conformally flat gradient steady
  {R}icci solitons}, ArXiv Preprint Server -- http://arxiv.org, 2009.

\bibitem{caochecatmanmaz}
H.-D. Cao, Q.~Chen, G.~Catino, C.~Mantegazza, and L.~Mazzieri, \emph{{B}ach
  flat gradient {R}icci solitons}, ArXiv Preprint Server -- http://arxiv.org,
  to appear on Calc. Var. Partial Differential Equations., 2011.

\bibitem{caosunzhang}
H.-D. Cao, X.~Sun, and Y.~Zhang, \emph{On the structure of gradient {Y}amabe
  solitons}, ArXiv Preprint Server -- http://arxiv.org, 2011.

\bibitem{caozhou}
H.-D. Cao and D.~Zhou, \emph{On complete gradient shrinking {R}icci solitons},
  J. Diff. Geom. \textbf{85} (2010), no.~2, 175--186.

\bibitem{mancat1}
G.~Catino and C.~Mantegazza, \emph{Evolution of the {W}eyl tensor under the
  {R}icci flow}, Ann. Inst. Fourier \textbf{61} (2011), no.~4, 1407--1435.

\bibitem{catmantmazz}
G.~Catino, C.~Mantegazza, and L.~Mazzieri, \emph{On the global structure of
  conformal gradient solitons with nonnegative {R}icci curvature}, Comm. Cont.
  Math. \textbf{14} (2012), no.~6, 1250045.

\bibitem{mmancatmazrim}
G.~Catino, C.~Mantegazza, L.~Mazzieri, and M.~Rimoldi, \emph{Locally
  conformally flat quasi--{E}instein manifolds}, J. Reine Angew. Math.
  \textbf{2013} (2013), no.~675, 181--189.

\bibitem{chowluni}
B.~Chow, P.~Lu, and L.~Ni, \emph{Hamilton's {R}icci flow}, Graduate Studies in
  Mathematics, vol.~77, American Mathematical Society, Providence, RI, 2006.

\bibitem{dasksesum}
P.~Daskalopoulos and N.~Sesum, \emph{The classification of locally conformally
  flat {Y}amabe solitons}, ArXiv Preprint Server -- http://arxiv.or, 2011.

\bibitem{fishmars}
A.~E. Fischer and J.~E. Marsden, \emph{Linearization stability of nonlinear
  partial differential equations}, Proc. Symp. Pure Math. \textbf{27} (1975),
  219--262.

\bibitem{hamilton1}
R.~S. Hamilton, \emph{Three--manifolds with positive {R}icci curvature}, J.
  Diff. Geom. \textbf{17} (1982), no.~2, 255--306.

\bibitem{hamilton5}
\bysame, \emph{The {R}icci flow on surfaces}, Mathematics and general
  relativity (Santa Cruz, CA, 1986), Contemp. Math., vol.~71, Amer. Math. Soc.,
  Providence, RI, 1988, pp.~237--262.

\bibitem{hamilton9}
\bysame, \emph{The formation of singularities in the {R}icci flow}, Surveys in
  differential geometry, Vol.~II (Cambridge, MA, 1993), Int. Press, Cambridge,
  MA, 1995, pp.~7--136.

\bibitem{HePetWylie}
C.~He, P.~Petersen, and W.~Wylie, \emph{On the classification of warped product
  {E}instein metrics}, ArXiv Preprint Server -- http://arxiv.org, 2010.

\bibitem{ivey1}
T.~Ivey, \emph{Ricci solitons on compact three--manifolds}, Differential Geom.
  Appl. \textbf{3} (1993), no.~4, 301--307.

\bibitem{jpb1}
J.-P.Bourguignon, \emph{Ricci curvature and {E}instein metrics}, Global
  differential geometry and global analysis ({B}erlin, 1979), Lecture Notes in
  Math., vol. 838, Springer, Berlin, 1981, pp.~42--63.

\bibitem{koba2}
O.~Kobayashi, \emph{A differential equation arising from scalar curvature}, J.
  Math. Soc. Japan \textbf{34} (1982), 665--675.

\bibitem{koiso1}
N.~Koiso, \emph{On rotationally symmetric {H}amilton's equation for
  {K}\"ahler-{E}instein metrics}, Recent topics in differential and analytic
  geometry, Adv. Stud. Pure Math., vol.~18, Academic Press, Boston, MA, 1990,
  pp.~327--337.

\bibitem{leltop}
C.~De Lellis and P.~M. Topping, \emph{Almost-{S}chur lemma}, Calc. Var.
  \textbf{43} (2012), 347--354.

\bibitem{perel1}
G.~Perelman, \emph{The entropy formula for the {R}icci flow and its geometric
  applications}, ArXiv Preprint Server -- http://arxiv.org, 2002.

\bibitem{perel3}
\bysame, \emph{Finite extinction time for the solutions to the {R}icci flow on
  certain three--manifolds}, ArXiv Preprint Server -- http://arxiv.org, 2003.

\bibitem{perel2}
\bysame, \emph{Ricci flow with surgery on three--manifolds}, ArXiv Preprint
  Server -- http://arxiv.org, 2003.

\bibitem{pw2}
P.~Petersen and W.~Wylie, \emph{On gradient {R}icci solitons with symmetry},
  Proc. Amer. Math. Soc. \textbf{137} (2009), no.~6, 2085--2092.

\bibitem{wangzhu}
X.~J. Wang and X.~H. Zhu, \emph{K\"ahler-ricci solitons on toric manifolds with
  positive first chern class}, Adv. Math. \textbf{188} (2004), no.~1, 87--103.

\bibitem{will}
C.~M. Will, \emph{The {C}onfrontation between {G}eneral {R}elativity and
  {E}xperiment}, Living Rev. Rel. \textbf{9} (2006).

\end{thebibliography}

\bigskip
\bigskip

\end{document}